\documentclass[sn-mathphys,Numbered]{sn-jnl}%

\usepackage{graphicx}%
\usepackage{multirow}%
\usepackage{amsmath,amssymb,amsfonts}%
\usepackage{amsthm}%
\usepackage{mathrsfs}%
\usepackage[title]{appendix}%
\usepackage{xcolor}%
\usepackage{textcomp}%
\usepackage{manyfoot}%
\usepackage{booktabs}%
\usepackage{algorithm}%
\usepackage{algorithmicx}%
\usepackage{algpseudocode}%
\usepackage{listings}%
\usepackage{nicefrac}%
\usepackage{pgfplots,tikz}%
\pgfplotsset{compat=1.18}%
\usepackage{mathtools}%
\usepackage{cool}%
\usepackage{cprotect}%
\usepackage{subfig}%
\usepackage{microtype}%
\newcommand{\tabitem}{\noindent\llap{\textbullet}~}
\theoremstyle{thmstyleone}%
\newtheorem{theorem}{Theorem}%
\newtheorem{proposition}[theorem]{Proposition}%

\theoremstyle{thmstyletwo}%
\newtheorem{remark}{Remark}%

\theoremstyle{thmstylethree}%

\begin{document}

\title[Efficient computation of the \texorpdfstring{$\operatorname{sinc}$}{sinc} matrix function]{Efficient computation of the \texorpdfstring{$\operatorname{sinc}$}{sinc} matrix function for the integration of second-order differential equations}

\author[1]{\fnm{Lidia} \sur{Aceto}}\email{lidia.aceto@uniupo.it}
\equalcont{These authors contributed equally to this work.}

\author*[2]{\fnm{Fabio} \sur{Durastante}}\email{fabio.durastante@unipit.it}
\equalcont{These authors contributed equally to this work.}

\affil[2]{\orgdiv{Dipartimento di Scienze e Innovazione Tecnologica}, \orgname{University of Eastern Piedmont}, \orgaddress{\street{Viale T. Michel, 11}, \city{Alessandria}, \postcode{15121}, \state{AL}, \country{Italy}}}

\affil*[2]{\orgdiv{Dipartimento di Matematica}, \orgname{University of Pisa}, \orgaddress{\street{Via F. Buonarroti, 1/C}, \city{Pisa}, \postcode{56127}, \state{PI}, \country{Italy}}}

\abstract{This work deals with the numerical solution of systems of oscillatory second-order differential equations which often arise from the semi-discretization in space of partial differential equations. Since these differential equations exhibit (pronounced or highly) oscillatory behavior, standard numerical methods are known to perform poorly. Our approach consists in directly discretizing the problem by means of Gautschi-type integrators based on $\operatorname{sinc}$ matrix functions. The novelty contained here is that of using a suitable rational approximation formula for the \texorpdfstring{$\operatorname{sinc}$}{sinc} matrix function to apply a rational Krylov-like approximation method with
suitable choices of poles. In particular, we discuss the application of the whole strategy to a
finite element discretization of the wave equation.}



\maketitle

\section{Introduction}
We consider the numerical solution of the system of multi-frequency oscillatory  second-order differential equations of the form
\begin{equation}\label{eq:theproblem}
\begin{cases}\mathbf{y}''(t) + A \mathbf{y}(t) = \mathbf{f}(t),\\
\mathbf{y}(t_0) = \mathbf{y}_0,\\
\mathbf{y}'(t_0) = \mathbf{y}_1,
\end{cases} 
\end{equation}
where $A \in \mathbb{R}^{n \times n}$ is a symmetric and positive semi-definite matrix implicitly containing the dominant frequencies of the problem. Usually, oscillatory differential equations arise from semi-discretization in space  of $d$-dimensional partial differential equations 
\[
\begin{cases}
u_{tt}(\mathbf{x},t) + \mathcal{A}(u(\mathbf{x},t)) = f(\mathbf{x},t), & (\mathbf{x},t) \in  \Omega \times (t_0,t_f], \, \Omega \subseteq \mathbb{R}^d, \\
u(\mathbf{x},t_0) = u_0(\mathbf{x}), & \mathbf{x} \in  \Omega\\
u_t(\mathbf{x},t_0) = v_0(\mathbf{x}), & \mathbf{x} \in  \Omega\\
\mathcal{B}(u(\mathbf{x},t)) = 0, & (\mathbf{x},t) \in \partial \Omega \times [t_0,t_f],
\end{cases}
\]
for $\mathcal{A}(\cdot)$ a differential operator with respect to the space variables -- involving either ordinary or fractional derivatives, and $\mathcal{B}(\cdot)$ the relevant boundary conditions.  The technique used to reduce partial differential equations to (large) ordinary differential equations is known as method of lines. In particular, the so-called \textit{longitudinal method of lines} separates the problem of discretization in space from the problem of evolution in time by using the intermediate step in which one discretizes only in space but maintains continuous time. Thus, the resulting semi-discrete problem is an initial value problem  of the form~\eqref{eq:theproblem}.

Several integration strategies start from the rewriting of~\eqref{eq:theproblem} as a first--order system. By introducing the  variable $\mathbf{z}(t) = [\mathbf{y}(t),\mathbf{y}'(t)]^T$ we are able to transform the second-order differential problem into the following system of first-order 
\[
\mathbf{z}'(t) + B\mathbf{z}(t) = \mathbf{g}(t), \quad \mathbf{z}(t_0) = \mathbf{z}_0
\]
where, for $O$ the all-zeros matrix and $I$ the identity matrix,
\[
B=\begin{bmatrix}
O & -I \\
A & O
\end{bmatrix}, \quad \mathbf{g}(t)= \begin{bmatrix}
    \mathbf{0} \\ \mathbf{f}(t)
\end{bmatrix}, \quad \mathbf{z}_0 =\begin{bmatrix}
    \mathbf{y}_0 \\ \mathbf{y}_1
\end{bmatrix}.
\]
Then, the %
solution can be obtained by reading the first row of the two-by-two block formula 
\begin{equation}\label{eq:exponentialintegrator}
\mathbf{z}(t) = \;\exp\left( -(t-t_0) B\right) \mathbf{z}_0   + \int_{t_0}^{t} \exp\left( -(t-\xi) B \right)   \mathbf{g}(\xi)
\,\mathrm{d}\xi.
\end{equation}
This choice is theoretically viable, and has been used many times in combination with Padé expansion for the matrix exponential, e.g., see~\cite{MR448947}. Furthermore, by applying suitable quadrature formulas to~\eqref{eq:exponentialintegrator} one produces the so-called exponential integrators schemes, for which we refer to the review~\cite{MR2652783}. To have efficient evaluations we require routines for computing products of matrix functions times a vector of Krylov-type. Nevertheless, we may have efficiency difficulties.  Indeed, if $A$ is symmetric one needs to work with a matrix that is similar to a skew-symmetric matrix and, as observed in~\cite{https://doi.org/10.48550/arxiv.2206.06909,MR1472203}, polynomial Krylov approaches for the matrix exponential have indeed a slower convergence rate than in the symmetric case. To overcome this drawback in~\cite{https://doi.org/10.48550/arxiv.2206.06909} the proposed approach turns to polynomial Krylov methods with restart, {where another possibility is to use a two-pass version of Lanczos for which there is no need to restart~\cite{MR2357621}, but one must take into account the loss of orthogonality and the need to operate with the transpose of the operator.}

Our approach in this work consists in directly discretizing the problem~\eqref{eq:theproblem} by means of Gautschi-type integrators~\cite{MR138200,MR1715573}; see Section~\ref{sec:integrators}. We focus on  the numerical integration scheme given by
\begin{equation*}
    \mathbf{y}_{n+1} - 2 \mathbf{y}_{n} + \mathbf{y}_{n-1} = h^2 
    \left( \operatorname{sinc } \frac{h \sqrt{A}}{2}  \right)^2
    (-A \mathbf{y}_n + \mathbf{f}_n),
\end{equation*}
with the \textit{unnormalized $\operatorname{sinc}$ function} defined  by
\pgfmathsetmacro\MathAxis{height("$\vcenter{}$")}
\begin{equation} \label{eq:sinc}
    {\displaystyle \operatorname {sinc} r}=   \left\{\begin{array}{cc}
        \displaystyle{\frac {\sin r}{r}} &  \mbox{for } r \neq 0\\
       \vspace{-3mm} &\\
         1& \mbox{for } r = 0
    \end{array} \right. = \begin{tikzpicture}[baseline={(0, 0.9cm-\MathAxis pt)}]
  \begin{axis}[
    height=1.42in,
    axis lines=middle,
    xmin=-13, xmax=13,
    ymin=-0.8, ymax=1.3,
    xtick=\empty,
    ytick={1},
    extra y tick style={
      tick label style={anchor=west, xshift=3pt},
    },
    function line/.style={
      red,
      thick,
    },
    single dot/.style={
      red,
      mark=*,
    },
  ]
    \addplot[function line, samples = 50, domain=-12.7:12.7] {sin(deg(x))/x};
    \addplot[single dot] coordinates {(0, 1)};
  \end{axis} 
  \end{tikzpicture}. 
\end{equation}
The novelty contained here is that of using a rational approximation formula for~\eqref{eq:sinc} to instead apply a rational Krylov-like approximation method with suitable choices of poles; see Sections~\ref{sec:rat-krylov} and~\ref{sec:rational-approx}. {Other approaches in the literature based on the use of polynomial Krylov methods are available, see~\cite{MR2235388} where the cosine formulation of the Gautschi scheme is used in combination with a polynomial Krylov method to integrate a finite element discretization of the Maxwell equation.}

Then, in Section~\ref{sec:mol-for-wave} we discuss the application of the whole strategy to a finite element discretization of the wave equation. Finally, in Section~\ref{sec:numerical-experiments} we perform a numerical exploration of the proposed approach.

\subsection{Notation} To simplify the reading of the paper, we report some notations we adopt throughout the paper. 
With $\|\cdot\|$ we indicate the vector and matrix $2$-norm, while with $\|\cdot\|_\Sigma$ we denote the function uniform norm (or sup norm) over the set $\Sigma$, i.e., $\|f\|_\Sigma = \sup \{ |f(s)|\,:\,s \in \Sigma \}$. 
For a generic matrix~$A$, we denote the field-of-values of $A$ by
\[
W(A) = \{ z \in \mathbb{C} \,: z = \mathbf{v}^H A \mathbf{v}, \; \mathbf{v} \in \mathbb{C}^n, \,\|\mathbf{v}\|=1\}.
\]

\section{Gautschi-type methods}
\label{sec:integrators}
In this section we recall the main features of time-stepping procedures that are generally known as the Gautschi-type methods~\cite{MR138200,MR1715573}. By making the usual position in which $\mathbf{y}_{n}$ is an approximation to the value $\mathbf{y}(t_0 + nh)$, $h = \nicefrac{(t_f - t_0)}{n_t}$, $n=0,1,\ldots,n_t$, one can introduce multi-step methods of a given \emph{trigonometric order} by looking at linear functionals $\mathfrak{L} \in \mathcal{C}^s$ {-~the
linear space of functions having $s$ continuous derivatives~-}of trigonometric order $p$, relative to period $T$, i.e., annihilating all trigonometric polynomials of degree $\leq p$ with period $T$ or, in other terms, such that
\[
\mathfrak{L}1 = \mathfrak{L} \cos\left( r \frac{2\pi}{T} t \right) = \mathfrak{L} \sin\left( r \frac{2\pi}{T} t \right) = 0, \quad r = 1,2,\ldots,p.
\]
In general, we can compare methods of trigonometric order $p$ with methods having algebraic order $2p$~\cite[p.~381]{MR138200}.
Among these are the extrapolation methods  of St\"{o}rmer-type of trigonometric order $p$ and with uniform step size $h$ which take the following form 
\begin{eqnarray}\label{eq:trigonometric-extrapolation}
\mathbf{y}_{n+1} + \alpha_{p,1}(v) \mathbf{y}_n + \alpha_{p,2}(v) \mathbf{y}_{n-1} = h^2 \sum_{\ell=1}^{2p-1} \beta_{p,\ell}(v) \mathbf{y}''_{n+1-\ell}, \qquad (v = \nicefrac{2 \pi h}{T}),
\end{eqnarray}
where the expressions of the $\alpha_{p,\ell}(\cdot)$ and $\beta_{p,\ell}(\cdot)$ power series are themselves in the seminal paper of Gautschi~\cite[Section~5]{MR138200}. 
In some cases they can be expressed in closed form as, for example, when $p=1.$ Indeed in the latter case they are the following 
\begin{equation} \label{eq:coef}
\alpha_{1,1}(v)=-2, \quad \alpha_{1,2}(v)=1, \quad \beta_{1,1}(v) =  ( \operatorname{sinc}{\nicefrac{v}{2}})^2.    
\end{equation}
Therefore, for the continuous problem \eqref{eq:theproblem} we can specify the extrapolation St\"{o}rmer scheme of trigonometric order $1$ as
\begin{equation}\label{eq:lubich_scheme}
    \mathbf{y}_{n+1} - 2 \mathbf{y}_{n} + \mathbf{y}_{n-1} = h^2 \psi(h^2A)(-A \mathbf{y}_n + \mathbf{f}_n),
\end{equation}
for $\psi(v^2)  = ( \operatorname{sinc}{\nicefrac{v}{2}})^2.$ 
Since this is a two-step scheme, we can construct its equivalent one-step formulation  on a staggered-grid as 
\begin{equation*}%
\begin{cases}
\mathbf{v}_{n + \nicefrac{1}{2}} = \mathbf{v}_n + \frac{h}{2} \psi(h^2 A)(-A \mathbf{y}_n + \mathbf{f}_n),\\
\mathbf{y}_{n+1} = \mathbf{y}_n + h \mathbf{v}_{n+\nicefrac{1}{2}},\\
\mathbf{v}_{n+1} = \mathbf{v}_{n+\nicefrac{1}{2}} + \frac{h}{2} \psi(h^2A)(-A\mathbf{y}_{n+1} + \mathbf{f}_{n+1})
\end{cases} \, \mathbf{v}_0 = \sigma(h^2A) \mathbf{y}_1,
\end{equation*}
where $\mathbf{v}_n$ represents the discretization of the velocity   $\mathbf{y}'(t_0+nh)$ and $\sigma(v^2) = \operatorname{sinc } v.$ If we concatenate the last equation coming from the previous time step with the first equation of the subsequent time step and we take as initial guess 
\begin{equation}\label{eq:lubich_scheme_compact_ini}
\mathbf{v}_{\nicefrac{1}{2}} = \sigma(h^2A) \mathbf{y}_1 + \frac{h}{2} \psi(h^2 A)(-A \mathbf{y}_0 + \mathbf{f}_0),
\end{equation}
we can simplify the scheme to
\begin{equation}\label{eq:lubich_scheme_compact}
\begin{cases}
    \mathbf{y}_{n+1} = \mathbf{y}_n + h \mathbf{v}_{n+\nicefrac{1}{2}},  & n \geq 0, \\
    \mathbf{v}_{n+\nicefrac{1}{2}} = \mathbf{v}_{n-\nicefrac{1}{2}} + h \psi(h^2 A)(-A \mathbf{y}_n + \mathbf{f}_n), & n > 0.
\end{cases}
\end{equation}
Therefore, the computational effort that we have to make in case we want to apply the scheme~\eqref{eq:lubich_scheme_compact} is the repeated computation of operations of the same type with {the $\psi$ function from}
\begin{equation}\label{eq:stormerfunctions}
    \psi(z) = \left(\operatorname{sinc} \nicefrac{\sqrt{z}}{2} \right)^2, \quad  %
    \sigma(z) = \operatorname{sinc} \sqrt{z}{,}
\end{equation}
{while the $\sigma$ is used once to calculate the initialisation value for the velocity.}
\begin{remark}\label{remark:stormer-verlet}
For $v \rightarrow 0,$ using \eqref{eq:sinc} it is easy to check that the coefficient $\beta_{1,1} (v)$ in~\eqref{eq:coef} tends to 1.
Therefore, the corresponding method reduces to the usual St\"ormer-Verlet-leapfrog method
\begin{equation} \label{SVL}
    \mathbf{y}_{n+1} - 2 \mathbf{y}_n + \mathbf{y}_{n-1} = h^2 (-A \mathbf{y}_n + \mathbf{f}_n).
\end{equation}
Arguments similar to those above give rise to the simplified scheme 
\[
\begin{cases}
    \mathbf{y}_{n+1} = \mathbf{y}_n + h \mathbf{v}_{n+\nicefrac{1}{2}}, & n \geq 0\\
     \mathbf{v}_{n+\nicefrac{1}{2}} = \mathbf{v}_{n-\nicefrac{1}{2}} + h (-A \mathbf{y}_n + \mathbf{f}_n), & n > 0\\
\end{cases}
\]
with initial guess 
\[
\mathbf{v}_{\nicefrac{1}{2}} = \mathbf{y}_1 + \frac{h}{2} (-A \mathbf{y}_0 + \mathbf{f}_0).
\]
This is the computationally most economic implementation, and numerically more stable than \eqref{SVL}; see  \cite[p. 472]{HNW1993}.
\end{remark}

\section{Rational Krylov methods}
\label{sec:rat-krylov}

An efficient way of performing the computations in \eqref{eq:lubich_scheme_compact_ini}-\eqref{eq:lubich_scheme_compact} is using subspace projection methods. In the following, we denote by $V_{k}$ an orthogonal matrix whose columns $\mathbf{v}_{1},\dots, \mathbf{v}_{k}$ span an arbitrary Krylov subspace $\mathcal{W}_{k}(A,\mathbf{v})$ of dimension~$k$. Then we can obtain an approximation of any $g(A;t)\mathbf{v}$ -- and analogously for any of the functions in~\eqref{eq:stormerfunctions} -- by
\begin{equation}
g(A;t)\mathbf{v} \approx \mathbf{g}_k = V_{k}g(V_{k}^{T}AV_{k};t)V_{k}^{T}\mathbf{v}. \label{eq:generic_krylov_approximation}
\end{equation}
For selecting between different methods of obtaining the approximation \eqref{eq:generic_krylov_approximation} we have to make suitable choices of the projection spaces $\mathcal{W}_{k}(A,\mathbf{v})$. In complete generality, if we can select a set of scalars -- called poles -- $\left\{ \zeta_{1},\dots,\zeta_{k-1}\right\} \subset \overline{\mathbb{C}}$ (the extended complex plane), that are not eigenvalues of~$A$, then we can define the polynomial
\begin{equation*}
q_{k-1}(z)=\prod_{j=1}^{k-1}(\zeta_{j}-z),
\end{equation*}
and consider as $\mathcal{W}_{k}(A,\mathbf{v})$ the rational Krylov subspace of order $k$ associated with $A$, $\mathbf{v%
}$ and $q_{k-1}$ defined by
\begin{equation}\label{eq:unseful_krylov}
\mathcal{Q}_{k}(A,\mathbf{v})=\left[ q_{k-1}(A)\right] ^{-1}\mathcal{%
K}_{k}(A,\mathbf{v}),
\end{equation}
for
\begin{equation*}
\mathcal{K}_{k}(A,\mathbf{v})=\operatorname{Span}\{\mathbf{v},A\mathbf{v}%
,\ldots ,A^{k-1}\mathbf{v}\}
\end{equation*}
the standard polynomial Krylov space. {The methods we consider here are rational, that is, they require the solution of a number of linear systems proportional to the number of poles chosen per step. If we compare them with methods based on polynomial approximation, which instead require only the use of matrix-vector products, these are therefore significantly more expensive. For there to be an advantage two conditions must be satisfied, first, the rational approximation must be able to capture the behavior of the function with a small number of poles---in other words, we must have the solution of few linear systems per step---second the solution of the required linear systems must be efficient. This is possible in several cases, for example when we have matrices that come from the discretization of elliptic operators, for which efficient multigrid methods exist and can be used both as solvers and as preconditioners in an inner-outer Krylov scheme, for other cases it is possible to use specialized direct methods~\cite{MR4712819,MR4266513}. Another case where there is usually some margin is when a high accuracy on the final solution is required. Due to the better approximation properties of rational functions compared to polynomials, this is usually the case. On the other hand, to obtain lower precision it is usually much more efficient to employ polynomial-type methods.}

The crucial point of the entire procedure is therefore the choice of the appropriate poles for the pair function and matrix under consideration. The principle to be guided by in the choice is that of the expression of the error committed using the approximation~\eqref{eq:generic_krylov_approximation} for a given Krylov space/set of poles.
\begin{theorem}[Near optimality, \cite{MR3095912,MR3666309}]\label{thm:Crouzeix}
Let $g$ be analytic in a neighborhood of a compact set $\Sigma \supseteq W(A)$ the field-of-value of $A$. Then the rational Krylov approximation~\eqref{eq:generic_krylov_approximation} defined using the space~\eqref{eq:unseful_krylov} satisfies 
\[
\| g(A;t)\mathbf{v} - \mathbf{g}_k \| \leq 2 C \|\mathbf{v}\| \min_{ r_k \in \mathbb{P}_{k-1}/q_{k-1} } \| g - r_k \|_\Sigma,
\]
with a constant $2 \leq C \leq 1 + \sqrt{2}$. If $A$ is Hermitian, the result holds even with $C = 1$ and $\Sigma \supseteq \Lambda(A) \cup \Lambda(V_k^T A V_k)$, for $\Lambda(\cdot)$ the spectrum.
\end{theorem}
The goal is therefore to choose the poles so as to obtain the best rational approximation of the function $g$ on the set $\Sigma$. For this purpose, a possible choice to obtain an upper bound is to fix a rational approximation for the function sought and to choose the poles of this approximation as poles of the method. With this in mind, in the next section we deal with obtaining these approximations in different ways.

\section{Four approximations to the \texorpdfstring{$\operatorname{sinc}$}{sinc} function}
\label{sec:rational-approx}

The heart of the approach is therefore that of finding an approximation for the $\operatorname{sinc}$ function~\eqref{eq:sinc} to either directly approximate the matrix function-vector product or to determine the poles $\{\zeta_j\}_j$ to build the Krylov space~\eqref{eq:unseful_krylov} and the related approximation~\eqref{eq:generic_krylov_approximation}. Such results can be obtained by using the expression of the $\operatorname{sinc}$ function in terms of the confluent hypergeometric function (Section~\ref{sec:rational-approx:pade}) or by inverting its Fourier transform (Section~\ref{sec:fourier-inversion}). While a closed form expression of the diagonal Padé approximant of the $\operatorname{sinc}$ function exist, see~\cite[p. 367]{MR367516}, it involves the computation of determinants of matrices whose entries are binomial coefficients. This task requires symbolical manipulations and does not produce an expression of the approximation error; we give a tabulation of few of them in Appendix~\ref{sec:sinc-pade-poles}.

\subsection{Padé-type approximants}\label{sec:rational-approx:pade}
For the following analysis we need to introduce the confluent hypergeometric function of the first kind. This can be defined by the generalized hypergeometric series  
\begin{equation*}%
    \Hypergeometric{1}{1}{a}{b}{x} = \sum_{n=0}^{+\infty} \frac{(a)_n}{(b)_n} \frac{x^n}{n!},
\end{equation*}
 where, as usual,  $(w)_s=w(w+1) \cdots (w+s-1)$ denotes the Pochhammer symbol. It is straightforward to verify that when $b=a$ this function coincides with the exponential function, i.e.,
 \begin{equation} \label{eq:expF11}
   e^x =\Hypergeometric{1}{1}{a}{a}{x}.
 \end{equation}
 Now, by virtue of the fact that 
 \begin{equation} \label{eq:sincexp}
    \operatorname{sinc} z = 
\frac{\sin z }{z}= \frac{e^{-iz}-e^{iz}}{-2iz},
\end{equation}
we may deduce a first rational approximation for the $\operatorname{sinc}$ function by the Padé approximants for the exponential.
Similarly, this can be done by considering 
the Padé approximants for the $\displaystyle \prescript{}{1}{F}_{1}\left(1;2 ; \cdot \right)$.  
In fact, we known that if $\mbox{Re } b > \mbox{Re } a > 0,$ the confluent hypergeometric function can be represented as an integral (see, e.g., \cite[eq. (1) p. 255]{Erdelyi-vol1})
\[
\Hypergeometric{1}{1}{a}{b}{x} = 
{\frac {\Gamma (b)}{\Gamma (a)\Gamma (b-a)}}\int _{0}^{1}e^{x t}t^{a-1}(1-t)^{b-a-1}\,dt.
\]
Setting $a=1,$ $b=2$ and recalling that for every positive integer $n,  \Gamma(n)=(n-1)!,$ we obtain
\begin{equation*} \label{phi1}
    \Hypergeometric{1}{1}{1}{2}{x} = 
\int _{0}^{1}e^{x t} \,dt 
= \frac{e^x-1}{x}.
\end{equation*}
Choosing $x=-2iz,$ from the previous and~\eqref{eq:sincexp} we get
\begin{equation} \label{sinF11}
    e^{iz} \Hypergeometric{1}{1}{1}{2}{-2iz} =  e^{iz} \frac{e^{-2iz}-1}{-2iz} =  \frac{e^{-iz}-e^{iz}}{-2iz}  = \operatorname{sinc} z.
\end{equation}
Alternatively, setting $x=\pm iz,$   we have
\begin{equation} \label{sinF11bis}
    \frac{1}{2} \Bigl(\Hypergeometric{1}{1}{1}{2}{iz}+\Hypergeometric{1}{1}{1}{2}{-iz} \Bigl)  = \operatorname{sinc} z.
\end{equation}
In all these three cases, rational approximations to the $\operatorname{sinc}$ function can be determined by using the results provided by Luke in \cite{Luke}  which concern the $[n/n]$-Padé approximant of $\displaystyle {}_{1}F_{1 }\left(1;b ;-x \right).$ As for the remainder, since the confluent hypergeometric functions can be linked to the so-called $\varphi$-functions by
\[
\displaystyle {}_{1}F_{1 }\left(1;j+1 ;x \right)= j! \,\varphi_j(x), \qquad j=0,1,2,\dots,
\]
we can use the error estimate deriving from the diagonal Padé approximant to $\varphi_j(x)$ and given in \cite[Lemma 2]{MR2492291} which is sharper than the one reported in \cite{Luke}.

\subsubsection{Padé approximants for the exponential function} \label{sec:exponential}
As already mentioned in \eqref{eq:expF11}, the exponential function can be expressed in terms of a confluent hypergeometric function as follows
\[
e^{-x} = \displaystyle {}_{1}F_{1 }\left(1;1 ;-x \right).
\]
Therefore, from \cite[eqs. (13)--(15) p. 15]{Luke} we can readily obtain its $[n/n]$-Padé approximant, i.e., 
\begin{equation} \label{eq:padexp}
     e^{-x} = \frac{G_n(-x)}{G_n(x)}  + S_n(x),
\end{equation}
where 
\begin{equation}\label{eq:gnfunction}
    G_n(x) = \frac{(2n)!}{n!} \displaystyle {}_{1}F_{1 }\left(-n;-2n ;x \right) 
\end{equation}
and the remainder is given in terms of modified Bessel functions
\begin{equation*}
S_n(x) = (-1)^{n+1} \pi e^{-x} \frac{I_{n+\nicefrac{1}{2}} (\nicefrac{x}{2})}{K_{n+\nicefrac{1}{2}} (\nicefrac{x}{2})} =
\frac{(-1)^{n+1} {(n!)^2}}{(2n)! (2n+1)!} x^{2n+1}
 + \mathcal{O} (x^{2n+2}).
\end{equation*}
\begin{remark}
Now, it is worth {observing} that the generalized Laguerre polynomial of degree $n$ is defined by suitable confluent hypergeometric function as follows (see, e.g., \cite[Eq.~(13.6.19)]{NIST:DLMF})
\begin{equation} \label{eq:pol-laguerre}
    {\displaystyle L_{n}^{(\alpha )}(x)\triangleq \binom{n+\alpha}{n} \Hypergeometric{1}{1}{-n}{\alpha+1}{x}}.
\end{equation}
While for $\alpha > - 1$ the $\{\displaystyle L_{n}^{(\alpha )}(x)\}_n$ are orthogonal polynomials on $[0,+\infty)$ with respect to the weight $x^\alpha e^{-x},$ for $\alpha \in \mathbb{C}$ this is no longer true{. That is,} they are no longer orthogonal and posses simple complex zeros; see, e.g., the discussion in~\cite{MR1858305}. This is the case of interest in the following computations.
\end{remark}
Then, we can express~\eqref{eq:gnfunction} as
\[
G_n(x) = \frac{(2n)!}{n!} \binom{-n-1}{n}^{-1} {\displaystyle L_{n}^{(-2n-1 )}(x)}.
\]
Consequently, \eqref{eq:padexp} becomes
\begin{equation*} \label{phi0}
    e^{-x} = \frac{\displaystyle L_{n}^{(-2n-1 )}(-x)}{\displaystyle L_{n}^{(-2n-1 )}(x)} + S_n(x).
\end{equation*}
This relation together with \eqref{eq:sincexp} gives rise to our first rational approximation for the $\operatorname{sinc}$ function 
\begin{eqnarray} 
    \operatorname{sinc} z %
    &=& -\frac{1}{2iz} \frac{\left(\displaystyle L_{n}^{(-2n-1 )}(-iz) \right)^2 - \left(\displaystyle L_{n}^{(-2n-1 )}(iz) \right)^2 }{\displaystyle L_{n}^{(-2n-1 )}(iz)\displaystyle L_{n}^{(-2n-1 )}(-iz)}
    + \hat S_n(z), \label{eq:sinc-exp-approximation}
\end{eqnarray}
where the remainder is given by
\begin{equation}
\begin{split}
\hat S_n(z) \triangleq & \frac{S_n(iz) -  S_n(-iz)}{-2iz} \approx \frac{n! n! }{(2n)! (2 n+1)!}z^{2 n} . \label{eq:sinc-exp-approximation-bound}
\end{split}
\end{equation}
Therefore, for each value of $n,$ the poles can be chosen as the zeros  of the polynomial $\displaystyle L_{n}^{(-2n-1 )}(iz)$ and its conjugate plus $z=0$, i.e.,
\[
\mathcal{E}_n = \{ \zeta \in \mathbb{C} \,:\, L_{n}^{(-2n-1 )}(i\zeta) = 0 \,\lor\,L_{n}^{(-2n-1 )}(-i\zeta) = 0 \} \cup \{0\}.
\]
We compare in Figure~\ref{fig:comparison_of_rational_approx0} the rational approximation in~\eqref{eq:sinc-exp-approximation} with the one obtained determining in a symbolic way the Padé expansion of the $\operatorname{sinc}$ function (see Appendix~\ref{sec:sinc-pade-poles}) and observe a good agreement of the two approximations. 

\begin{figure}[htbp]
    \centering
    \subfloat[Approximation~\eqref{eq:sinc-exp-approximation} with $n=5$.]{\includegraphics[width=0.45\columnwidth]{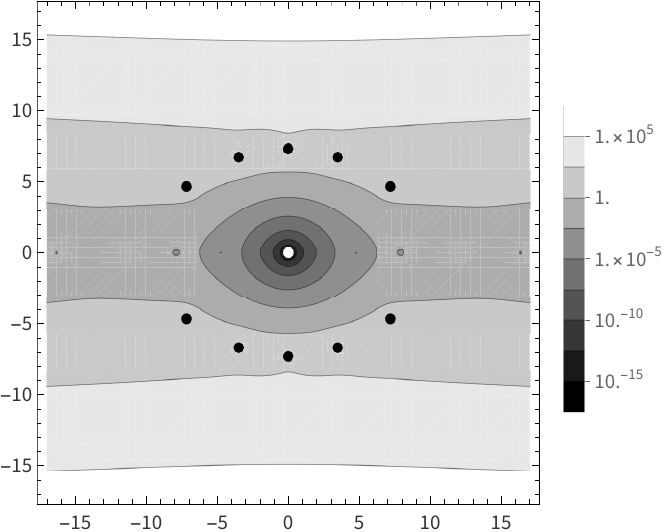}}
    \subfloat[{$[10/10]$}-Padé approximation\label{fig:reference-pade}]{\includegraphics[width=0.45\columnwidth]{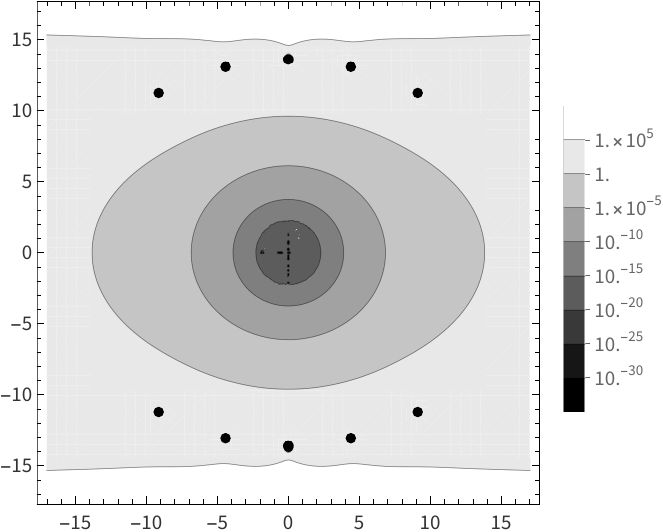}}
    \caption{Absolute error for the approximation~\eqref{eq:sinc-exp-approximation} and for the $[10/10]$-Padé approximation of the $\operatorname{sinc}$ function computed numerically in the $[-17,17]\times[-17i,17i]$ region. The color-map is $\log_{10}$-scale and the bold black dots represent the poles in the two cases; for the approximation~\eqref{eq:sinc-exp-approximation} the pole in $0$ is denoted in white to make it visible against its surrounding.}
    \label{fig:comparison_of_rational_approx0}
\end{figure}

\subsubsection{Padé approximants for \texorpdfstring{$\Hypergeometric{1}{1}{1}{2}{\cdot}$}{}} \label{sec:confhyperfunc}
Following the results given by Luke in \cite[Section 3]{Luke} and again the higher order term that can be obtained from \cite[Lemma 2]{MR2492291}, we can easily obtain the $[n/n]$-Padé approximant to the function $\displaystyle \prescript{}{1}{F}_{1}\left(1;2 ;-x \right).$ Thus, we write 
\begin{eqnarray} \label{eq:pade-gen}
   \Hypergeometric{1}{1}{1}{2}{-x} = \frac{\mathcal{A}_n(x)}{\mathcal{B}_n(x)}  + R_n(x),
\end{eqnarray}
where 
 \begin{eqnarray}
 {\mathcal{A}_n(x)} &{=}& {x^n\frac{\Gamma(n+2)}{\Gamma(2n+2)}\sum_{k=0}^n \frac{(-n)_k(n+2)_k }{(2)_k k!} \Hypergeometric{3}{1}{-n+k,n+2+k,1}{1+k}{-x^{-1}}},\nonumber   \\
 \mathcal{B}_n(x)   &=&  x^n\frac{\Gamma(n+2)}{\Gamma(2n+2)}  \Hypergeometric{2}{0}{-n,n+2}{--}{-x^{-1}}, \label{qnpade} \\
 R_n(x) &=& (-1)^{n+1} \left(\frac{n!}{\sqrt{2}(2n+1)!}\right)^2 x^{2n+1} + \nonumber \\ &\phantom{=}& (-1)^{n+1} \frac{n+1}{2n+3} \left(\frac{n!}{\sqrt{2}(2n+1)!}\right)^2  x^{2n+2} + \mathcal{O} (x^{2n+3}). \label{eq:resto}
\end{eqnarray}
{Note that the explicit form of the polynomial $\mathcal{A}_n(x)$ is useless for our analysis, and is not needed for the computation; we reported it for the sake of completion.} Since we are interested in understanding which are the zeros of the denominator of this $[n/n]$-Padé approximant, we now focus on $\mathcal{B}_n(x).$ By virtue of the fact that (see \cite[eq. (3) p. 257]{Erdelyi-vol1})
\[
\Hypergeometric{2}{0}{-n,n+2}{--}{-x^{-1}} = x^{-n} \Psi(-n,-2n-1,x)
\]
we can write
\[
\mathcal{B}_n(x)  =\frac{\Gamma(n+2)}{\Gamma(2n+2)} \Psi(-n,-2n-1,x).
\]
Using \cite[eq. (7) p. 257]{Erdelyi-vol1} we find
\[
 \Psi(-n,-2n-1,x) = \frac{\Gamma(2n+2)}{\Gamma(n+2)} \Hypergeometric{1}{1}{-n}{-2n-1}{x}.
\]
Consequently,
\begin{equation*} %
    \mathcal{B}_n(x)  =\displaystyle \prescript{}{1}{F}_{1}\left(-n;-2n-1 ;x\right).
\end{equation*}
Taking into account the formula \eqref{eq:pol-laguerre}
we have
\[
    \mathcal{B}_n(x) =   \binom{-n-2 }{n}^{-1}
    {\displaystyle L_{n}^{(-2n-2 )}(x)}. 
\]
Finally, noting that
\[
 \binom{-n-2}{n} = (-1)^n  \binom{2n+1}{n}
\]
it is immediate to get 
\begin{equation*} %
    \mathcal{B}_n(x) =  (-1)^n  \binom{2n+1}{n}^{-1} {\displaystyle L_{n}^{(-2n-2 )}(x)}. 
\end{equation*}
Therefore, the above equation together with  \eqref{eq:pade-gen} and \eqref{eq:resto} leads to the following rational approximation (see \eqref{sinF11})
\begin{equation}\label{eq:approx-pade-laguerre}
   \operatorname{sinc } z = (-1)^n  \displaystyle \binom{2n+1}{n} \frac{\mathcal{A}_n(2iz)}{ {\displaystyle L_{n}^{(-2n-2 )}(2iz)}}  e^{iz} + \hat R_n(z),
\end{equation}
where
\begin{equation}\label{eq:approx-pade-laguerre-bound}
\begin{split}
   \hat R_n(z) \triangleq & R_n(2iz)  e^{iz} \approx  - i  \left(\frac{2^{n} n!}{(2n+1)!}\right)^2 z^{2 n+1} e^{i z}.
\end{split} 
\end{equation}
Given this rational approximation, the poles can then be chosen as the zeros of the polynomials $L_{n}^{(-2n-2 )}(2iz)$ for different values of $n$, i.e., the set
\[
\mathcal{L}_{n} = \{ \zeta \in \mathbb{C}:\, L_n^{(-2n-2)}(2i\zeta) = 0 \}.
\]
In Figure~\ref{fig:comparison_of_rational_approx1} we compare the rational approximation in~\eqref{eq:approx-pade-laguerre} with the one obtained determining in a symbolic way the Padé expansion of the $\operatorname{sinc}$ function; we observe a good agreement between the two. 
\begin{figure}[htbp]
    \centering
    \subfloat[Approximation~\eqref{eq:approx-pade-laguerre} with $n=10$.]{\includegraphics[width=0.45\columnwidth]{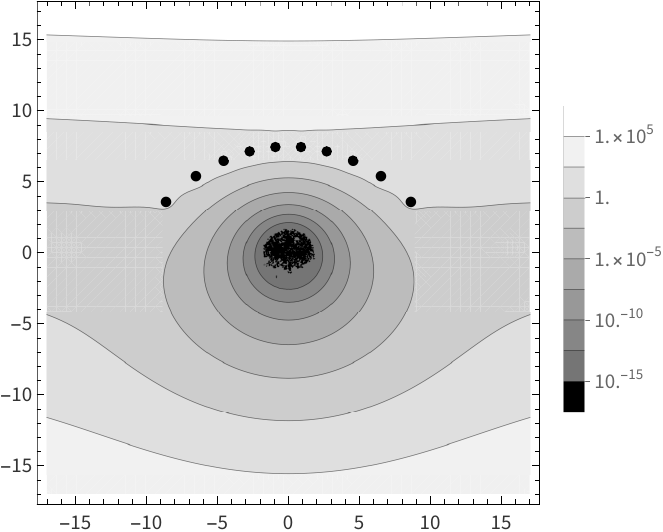}}
    \subfloat[{$[10/10]$}-Padé approximation]{\includegraphics[width=0.45\columnwidth]{PadeRationalApprox}}
    
    \caption{Absolute error for the approximation~\eqref{eq:approx-pade-laguerre} and for the $[10/10]$-Padé approximation of the $\operatorname{sinc}$ function computed numerically in the $[-17,17]\times[-17i,17i]$ region. The color-map is $\log_{10}$-scale and the bold black dots represent the poles.}
    \label{fig:comparison_of_rational_approx1}
\end{figure}
Nevertheless, we discover a lack of symmetry of the poles obtained from the expansion \eqref{eq:approx-pade-laguerre}, symmetry that should be inherited from the parity of the $\operatorname{sinc}$ function. This does not happen when rewriting the $\operatorname{sinc}$ function as in~\eqref{sinF11bis}. Indeed, a repeated application of~\eqref{eq:pade-gen}-\eqref{eq:resto} leads to
\begin{equation}\label{eq:approx-pade-laguerre1}
\begin{split}
   \operatorname{sinc } z = & \frac{(-1)^n}{2}  \displaystyle \binom{2n+1}{n}  \frac{\mathcal{A}_n(iz)  {\displaystyle L_{n}^{(-2n-2 )}(-iz)} + \mathcal{A}_n(-iz)  {\displaystyle L_{n}^{(-2n-2 )}(iz)}}{ {\displaystyle L_{n}^{(-2n-2 )}(iz) \, \displaystyle L_{n}^{(-2n-2 )}(-iz)}} + \tilde{R}_n(z), 
\end{split}
\end{equation}
with
\begin{equation}\label{eq:approx-pade-laguerre1-err}
\begin{split}
   \tilde{R}_n(z) = & -\frac{ ((n+1)!)^2 }{(2 n+1)! (2 n+3)!}z^{2 n+2}+ \mathcal{O}(z^{2n+3}).
\end{split}
\end{equation}
This suggests considering as poles the set
\[
\overline{\mathcal{L}}_{n} =  \{ \zeta \in \mathbb{C} \,:\, L_n^{(-2n-2)}(i\zeta) = 0 \,\lor\, 
L_n^{(-2n-2)}(-i\zeta) = 0 \}.
\]

A depiction of the scalar bounds for these rational approximations can be found in Appendix~\ref{sec:appendix-bound}. 

We conclude this section by applying the previous analysis to the matrix case and summarizing it in the following result.
\begin{proposition}\label{pro:how-many-poles}
Let $A$ be a symmetric and positive semi-definite matrix and $L_{n}^{(\alpha)}(x)$ the generalized Laguerre polynomial of degree $n$ defined in \eqref{eq:pol-laguerre}. Then, the following rational approximations to the $\operatorname{sinc}$ matrix function  can be derived:
\begin{itemize}
    \item[(i)] using the Padé approximant to the exponential function and
setting 
\[
P_n(A):=\left( L_{n}^{(-2n-1 )}(iA)\right)^{-1}  L_{n}^{(-2n-1 )}(-iA),
\]
from \eqref{eq:sinc-exp-approximation} we have 
\[
\operatorname{sinc}(A) \approx  -\frac{1}{2i} A^{-1}\left[ P_n(A) - (P_n(A))^{-1}\right]:=E_n(A),
\]
with the following bound for the error:
\begin{eqnarray} 
    \|\operatorname{sinc}(A) -E_n(A)\| & \le &  %
    2 \left\| (2n+1) \left(\frac{ \, n!}{(2n+1)!}\right)^2
    z^{2 n} \right\|_\Sigma; \label{eq:sinc-exp-approximation-mat} %
\end{eqnarray}
\item[(ii)] using the Padé approximant for 
\texorpdfstring{$\Hypergeometric{1}{1}{1}{2}{-2iz}$}{}, from  \eqref{eq:approx-pade-laguerre} we have 
\[
\operatorname{sinc}(A) \approx  (-1)^n  \displaystyle \binom{2n+1}{n} \mathcal{A}_n(2iA) \left(L_{n}^{(-2n-2 )}(2iA) \right)^{-1} e^{iA}:=F_n(A),
\]
with the following bound for the error:
\begin{eqnarray} 
    \|\operatorname{sinc}(A) -F_n(A)\| & \le &  %
    2 \left\| 2^{2n} \left(\frac{ n!}{(2n+1)!}\right)^2 z^{2 n+1} \right\|_\Sigma. %
    \label{eq:approx-pade-laguerre-mat}
\end{eqnarray}
For the symmetrized version from~\eqref{eq:approx-pade-laguerre1} we call $\Tilde{F}_n(A)$ the evaluation of the approximation in $A$ thus obtaining
\begin{equation}\label{eq:approx-pade-laguerre1-mat}
\|\operatorname{sinc}(A) -\Tilde{F}_n(A)\| \leq  2 \left\| \frac{n+1}{4n+6} \left(\frac{ n!}{(2n+1)!}\right)^2 z^{2 n+2} \right\|_\Sigma.
\end{equation}
\end{itemize}
\end{proposition}

\begin{proof}
From Theorem~\ref{thm:Crouzeix}, the bounds are essentially the scalar bounds in the equations~\eqref{eq:sinc-exp-approximation-bound} for~\eqref{eq:sinc-exp-approximation}, \eqref{eq:approx-pade-laguerre-bound} for~\eqref{eq:approx-pade-laguerre}, and~\eqref{eq:approx-pade-laguerre1-err} for~\eqref{eq:approx-pade-laguerre1} respectively, a part from a multiplicative factor $2 C = 2$ due to $A$ being symmetric. 
\end{proof}

It should be noted that, given the simplicity of computing the numerators and denominators of the rational approximations discussed here, it is also feasible to use them directly for the computation of the product $\operatorname{sinc}(A)\mathbf{v}$. This costs the same number of solutions of linear systems with shifted coefficient matrix as the case of the rational Krylov method, however with the same right-hand term and -~in principle~- the possibility of solving them simultaneously.

\subsection{Exponential sums from the inverse Fourier transform}\label{sec:fourier-inversion}
We recall that the Fourier transform of the $\operatorname{sinc}(r)$ function can be expressed~as
\[
\mathcal{F} \left\lbrace \operatorname{sinc}(r) \right\rbrace(k) = \frac{1}{\sqrt{2 \pi }}\int _{-\infty }^{\infty } 
\operatorname{sinc} (r) e^{i r k }\, \mathrm{d}r = \frac{1}{2} \sqrt{\frac{\pi }{2}} (\operatorname{sgn}(1-k)+\operatorname{sgn}(1+k)),
\]
where 
\pgfmathsetmacro\MathAxis{height("$\vcenter{}$")}
\[
\operatorname{sgn}(r) = \begin{cases}
    +1, & r > 0,\\
    \phantom{+}0, & r = 0,\\
    -1,& r < 0,
\end{cases} = \begin{tikzpicture}[baseline={(0, 1cm-\MathAxis pt)}]
  \begin{axis}[
    height=1.42in,
    axis lines=middle,
    xmin=-3, xmax=3,
    ymin=-1.5, ymax=1.5,
    xtick=\empty,
    ytick={0, 1},
    extra y ticks={-1},
    extra y tick style={
      tick label style={anchor=west, xshift=3pt},
    },
    function line/.style={
      red,
      thick,
      samples=2,
    },
    single dot/.style={
      red,
      mark=*,
    },
    empty point/.style={
      only marks,
      mark=*,
      mark options={fill=white, draw=black},
    },
  ]
    \addplot[function line, domain=\pgfkeysvalueof{/pgfplots/xmin}:0] {-1};
    \addplot[function line, domain=0:\pgfkeysvalueof{/pgfplots/xmax}] {1};
    \addplot[single dot] coordinates {(0, 0)};
    \addplot[empty point] coordinates {(0, -1) (0, 1)};
  \end{axis}
\end{tikzpicture}.
\]
Thus we can approximate $\operatorname{sinc}(A)\mathbf{v}$ by approximating the integral of the inverse Fourier transform
\[
\operatorname{sinc}(A)\mathbf{v} = \frac{1}{2} \int_{-1}^{1} \exp(-i k A )\mathbf{v}\,\mathrm{d}k,
\]
if we then call $\{ \omega_p,\ell_p \}_{p=1}^{\nu}$ the Gauss-Legendre quadrature points in the interval $[-1,1]$ we can approximate it as the \emph{exponential sum}
\begin{equation}\label{eq:fourierapprox}
    \operatorname{sinc}(A)\mathbf{v} \approx \frac{1}{2}\sum_{p=1}^{\nu} \omega_p \exp(- i \ell_p A )\mathbf{v}.
\end{equation}
Another reasonable way to approximate this integral would be to use instead the {C}lenshaw-{C}urtis quadrature formula; this choice often obtains results comparable to that of Gauss-Legendre, as discussed in~\cite{MR2403058}. In the present case, numerical tests have shown us a better convergence behavior for the Gauss formula. In particular, for the latter we can show the following error bound.
\begin{proposition}\label{pro:sinc-fourier}
The error for the approximation based on~\eqref{eq:fourierapprox} for $A$ symmetric and positive semi-definite is bounded by
\[
\left\| \operatorname{sinc}(A) -  \frac{1}{2}\sum_{p=1}^{\nu} \omega_p \exp(- i \ell_p A ) \right\| \leq \frac{\pi}{(2\nu)!} \left(\frac{\rho(A)}{2}\right)^{2\nu},
\]
for $\rho(A)$ the spectral radius of $A$.
\end{proposition}

\begin{proof}
Since $\exp(-i z k)$ has $2\nu$ continuous derivative in $[-1,1]$, the error of the integral approximation can be expressed as \cite[Chap. 5, Page~146]{10.5555/61926}
\[
Q_\nu = \frac{2^{2\nu+1} (\nu!)^4}{(2\nu+1)( (2\nu)! )^3}\max_{k \in [-1,1]} \left\lvert \frac{d^{2\nu}}{d k^{2\nu} } \exp( -i k z  ) \right\rvert,
\]
that is
\[
Q_\nu \leq \frac{2^{2\nu+1} (\nu!)^4}{(2\nu+1)( (2\nu)! )^3} |z|^{2\nu},
\]
and the statement follows from
\[
\lim_{\nu\to +\infty }\frac{2^{4 \nu+1} (\nu!)^4}{(2 \nu)! (2 \nu+1)!} = \pi. \qedhere
\]
\end{proof}

To compute the sum~\eqref{eq:fourierapprox} we have to compute $\nu$ matrix-exponential times vector products. The situation is analogous to the one encountered in \emph{exponential integrators} for first-order differential equations~\eqref{eq:exponentialintegrator}. To this end we can employ the rational approximation to the exponential function discussed in Section~\ref{sec:exponential} to generate poles for computing the rational Krylov subspace $\mathcal{Q}_k(-A,\mathbf{v})$ (see again the discussion in Section~\ref{sec:rat-krylov}), and then approximate the exponential sum~\eqref{eq:fourierapprox} as
\[
\operatorname{sinc}(A)\mathbf{v} \approx \frac{1}{2}\sum_{p=1}^{\nu} \omega_p V_k \exp( i \ell_p A_k) V_k^T \mathbf{v}, \quad A_k = - V_k^T A V_k,
\]
that can be easily adapted to the computation of~\eqref{eq:stormerfunctions} by discharging the computation of the square roots on the projected matrix $A_k$.

\begin{remark}
Other viable approaches for performing the computation of the matrix-exponential vector products in~\eqref{eq:fourierapprox} could be the polynomial strategy discussed in~\cite{MR2785959} that combines the usage of scaling-and-squaring techniques together with a Taylor expansion of the exponential function, or the rational expansion using the {C}arath\'{e}odory-{F}ej\'{e}r approximation in~\cite{MR2494943}. Nevertheless, since we need to approximate the matrix-vector product with respect to the functions in~\eqref{eq:stormerfunctions} the polynomial approach would require dealing with the $\sqrt{A}$. On the other hand, the {C}arath\'{e}odory-{F}ej\'{e}r rational approximation is designed to be uniformly optimal on negative real axis, while we need it on the imaginary one. This implies that to have sufficient accuracy, we may have to use a large number of poles to widen the region of fast convergence in the complex plane. Both alternatives are implemented in the code distributed with the paper as discussed in Section~\ref{sec:numerical-experiments}.
\end{remark}

The proposed method can be used directly for the computation of the $\sigma(z)$ function in~\eqref{eq:stormerfunctions} for the initial velocity in~\eqref{eq:lubich_scheme}, on the other hand, a slightly adapted version is needed for the computation of the function $\psi(z)$ in~\eqref{eq:stormerfunctions} since we need to compute the matrix-vector product with respect to the square of the $\operatorname{sinc}$ function. This can be done again by inverting the related Fourier transform, in fact:
\[
\mathcal{F} \left\lbrace \operatorname{sinc}^2(r) \right\rbrace(k) = \frac{1}{4} \sqrt{\frac{\pi }{2}} ((k-2) \text{sgn}(k-2)-2 k \text{sgn}(k)+(k+2) \text{sgn}(k+2)),
\]
and thus
\[
\begin{split}
\operatorname{sinc}^2(A)\mathbf{v} = & \frac{1}{8} \left[ \int_{-2}^0 (2 k+4) \exp \left(-i k A\right) \, \mathrm{d}k+\int_0^2 (4-2 k) \exp \left(-i k A\right) \, \mathrm{d}k\right]\mathbf{v}, \\
= & \frac{1}{8} \int_{-2}^{0} (2 k + 4) \left(\exp \left(-i k A \right)\mathbf{v}+\exp \left(i k A \right) \mathbf{v}\right) \, \mathrm{d}k.
\end{split}
\]
By the same procedure, this integral can be approximated combining the Gauss-Legendre quadrature points in the interval $[-2,0]$ with the rational Krylov method based on the Padé poles for the exponential, thus obtaining
\[
\begin{split}
\operatorname{sinc}^2(A)\mathbf{v} \approx & \frac{1}{8} V_k \left( \sum_{p=1}^{\nu} \omega''_p (2 \ell''_p + 4) \exp(i \ell''_p A_k) (V^T_k \mathbf{v})  \right. \\ & + \left. \sum_{p=1}^{\nu} \omega''_p (2 \ell''_{\nu-p+1}+4) \exp(-i \ell''_p A_k) (V^T_k \mathbf{v}) \right), 
\end{split}
\]
for $A_k = -V_k^T A V_k$, and $V_k$ a matrix whose columns span the rational Krylov subspace for $A$ and $\mathbf{v}$ with respect to the rational approximation to the exponential function discussed in Section~\ref{sec:exponential}. Note that to write the previous one we have also exploited the symmetry properties  of the Gauss-Legendre weights and nodes, i.e., that the nodes and weights $\{ \omega'_p,\ell'_p \}_{p=1}^{\nu}$ in $[0,2]$ are such that $\omega'_p = \omega''_p$ and $\ell'_p = - \ell''_{\nu-p+1}$, where $\{ \omega''_p,\ell''_p \}_{p=1}^{\nu}$ are the Gauss-Legendre weights and nodes on the interval~$[-2,0]$.

{We conclude this section by summarizing all the discussed approaches in Table~\ref{tab:summary} with their advantages and disadvantages. All the different choices are implemented in the distributed code and can be tested accordingly.}
\begin{table}[htbp]
    \centering
    \begin{tabular}{ccp{\dimexpr 0.3\linewidth-2\tabcolsep}p{\dimexpr 0.3\linewidth-2\tabcolsep}}
    \toprule
        Approximation & Error & Advantages & Disadvantages  \\
    \midrule
       \eqref{eq:sinc-exp-approximation}  & \eqref{eq:sinc-exp-approximation-mat} & \tabitem $\mathcal{E}_n$ poles easily computable, \newline\tabitem \textit{a-priori} error estimate. & \tabitem requires invertible matrix.\\
       \midrule
       \eqref{eq:approx-pade-laguerre} & \eqref{eq:approx-pade-laguerre-mat} & \tabitem $\mathcal{L}_n$ poles easily computable,\newline\tabitem \textit{a-priori} error estimate. & \tabitem rational approximation is not even.\\
       \midrule
       \eqref{eq:approx-pade-laguerre1} & \eqref{eq:approx-pade-laguerre1-mat} & \tabitem $\overline{\mathcal{L}}_n$ poles easily computable,\newline\tabitem \textit{a-priori} error estimate,\newline\tabitem best error-computational cost ratio. & \\  
       \midrule
       \S\ref{sec:fourier-inversion} & Prop.~\ref{pro:sinc-fourier} &\tabitem same quadrature weights and nodes for~$\operatorname{sinc}\,r$ and $\operatorname{sinc}^2\,r$. &\tabitem error estimate depends on spectral radius \\
    \bottomrule
    \end{tabular}
    \caption{{We collect here the summary information of the different approaches we have introduced. In particular, we briefly summarize the advantages and disadvantages of the different choices.}}
    \label{tab:summary}
\end{table}

\section{A method of lines for the linear wave equation}
\label{sec:mol-for-wave}

Let us consider the acoustic wave equation
\begin{equation}\label{eq:acoustic-wave}
\begin{cases}
u_{tt} - \Delta u = f, & \text{ in } \Omega \times (0,T],\\
u(\mathbf{x},0) = u_0(\mathbf{x}), &  \;\mathbf{x} \in \Omega,\\
u_t(\mathbf{x},0) = v_0(\mathbf{x}), & \;\mathbf{x} \in \Omega,\\
u = g_D, & \text{ on } \Gamma_D \times [0,T],\\
\nabla u \cdot \hat{\mathbf{n}} = g_N, & \text{ on } \Gamma_N \times [0,T],\\
\end{cases}  
\end{equation}
with $T > 0$, $\Omega \subset \mathbb{R}^d$, $d=2,3$, an open bounded domain with Lipschitz continuous boundary $\partial \Omega = \Gamma_D \cup \Gamma_N$, $\Gamma_D \cap \Gamma_N = \emptyset$, and normal vector $\hat{\mathbf{n}}$. To arrive at a system of second-order differential equations of the form~\eqref{eq:theproblem} we apply a space semi-discretization using a Finite Element Method (FEM). Let $H \triangleq \mathbb{L}^2(\Omega)$ be the space square integrable function for which we denote the scalar product as $\langle\cdot,\cdot\rangle$, and the corresponding norm $\|\cdot\|_H$, then we denote~by 
\[
\begin{split}
    \mathbb{H}^1(\Omega) = & \{ v \in H \,:\, \nicefrac{\partial v}{\partial x_i} \in H, \, i=1,\ldots,d \},\\
    \mathbb{H}^{1}_{\Gamma_D}(\Omega) = & \{v \in \mathbb{H}^{1}(\Omega) \,:\, v = 0 \text{ on } \Gamma_D \} \triangleq V,
\end{split}
\]
together with its dual $\mathbb{H}^{-1}_{\Gamma_D}(\Omega) \triangleq V^*$. The semi-discrete weak formulation of~\eqref{eq:acoustic-wave} can therefore be expressed as
\begin{equation}\label{eq:variational-form}
\begin{split}
\langle u_{tt}(\cdot,t),w \rangle + a(u(\cdot,t),w) = \langle f(\cdot,t),w \rangle + \langle g_N,w \rangle, &\qquad \forall w \in V,\\
\langle u(\cdot,0),w \rangle = \langle u_0,w \rangle, &\qquad \forall w \in H,\\ 
\langle u_t(\cdot,0),w \rangle = \langle v_0,w \rangle, &\qquad \forall w \in H, 
\end{split}
\end{equation}
with
\[
a(u,w) = \int_{\Omega} \nabla u \cdot \nabla w \,\mathrm{d}\mathbf{x}, \qquad \forall\,u,w \in V,
\]
where we use the space-time function spaces 
\[
\mathbb{L}^2(0,T;X) = \left\lbrace \phi\,:\,(0,T) \rightarrow X \,\text{s.t.}\, \int_{0}^{T} \|\phi(t)\|_X^2\,\mathrm{d}t < \infty \right\rbrace,
\]
in which we consider any $\phi(t) \triangleq \phi(\cdot,t)$ a function of the space variable only for fixed values of $t$.
We must now fix the spaces of discrete approximation to obtain~\eqref{eq:theproblem} as a Galerkin approximation of the variational formulation~\eqref{eq:variational-form}{; see, e.g., \cite[Chapter 3]{MR3526765}.}

\subsection{FEM approximation spaces}\label{sec:fem-spaces}

For a set $\Omega \subseteq \mathbb{R}^2$ we consider a triangulation $\mathcal{T}_h$ with maximum mesh edge length $h$ of the domain $\Omega$ and the space of linear finite elements, i.e., the space of piecewise linear continuous finite element, for which we recall that the degrees of freedom (DoFs) needed to uniquely identify a function in $\mathcal{V}_h$ are its values at the vertices,
\[
\mathcal{V}_h =   \;\{ v \in \mathbb{H}^1(\Omega) \,:\, \forall\,\kappa \in \mathcal{T}_h, \; \left.v\right|_{\kappa} \in \mathbb{P}^1[\mathbf{x}] \} \\ =  \operatorname{Span}\{ \phi_{j}(\mathbf{x})\}_{j=1}^{N_{\text{DoFs}}} \subseteq \mathcal{V} \cap \mathcal{C}([0,T];V).
\]

The discrete version of~\eqref{eq:variational-form} can then be written in $\mathcal{V}_h$ as
\[
M \mathbf{u}''(t) + K \mathbf{u}(t) = \mathbf{F}(t) + \mathbf{G}(t),
\]
with { $M$, $K$ the usual mass and stiffness matrices, and $\mathbf{F}(t)$, $\mathbf{G}(t)$ contain the source and Neumann boundary information.}
To impose Dirichlet boundary conditions we use the so-called null-space technique, that is, we eliminates Dirichlet conditions from the problem by operating on the matrix. We build a matrix $Z$ spanning the null-space of the columns of the matrix representing the Dirichlet condition equations, i.e., we restrict the matrices as
\[
M_c = Z^T M Z, \qquad  K_c = Z^T K Z, \qquad \mathbf{F}_c(t) = Z^T\left( \mathbf{F}(t) + \mathbf{G}(t) - K \mathbf{g}_D \right),
\]
where $\mathbf{g}_D$ is the function evaluating the Dirichlet boundary conditions at time $t$ on the relevant degree of freedoms and taking value zero elsewhere. We can now rewrite the scheme~\eqref{eq:lubich_scheme_compact} for this specific case using these matrix expressions.

\section{Numerical experiments}
\label{sec:numerical-experiments}

In the following three sections we first consider the different choices of poles for computing the $\operatorname{sinc}$ function of a matrix and the approach based on exponential sums; Section~\ref{sec:numerical-experiments:poles}. Then we consider a synthetic example of a second-order differential equation of the type~\eqref{eq:theproblem} to show how error analysis allows us to choose the number of poles we need to maintain the convergence order of the method~\eqref{eq:lubich_scheme_compact}; Section~\ref{sec:a-synthetic-example}. Finally, we consider the application of the method to the solution of the discretized finite element version of the wave equation~\eqref{eq:acoustic-wave}; Section~\ref{sec:numerical-experiments:waves}. The code for reproducing the experiments is available in the GitHub repository \href{https://github.com/Cirdans-Home/rationalsecondorder}{github.com/Cirdans-Home/rationalsecondorder}. All the numerical experiments have been run on a Linux laptop with an Intel\textsuperscript{\textregistered} Core\textsuperscript{\texttrademark} i7-8750H CPU at 2.20GHz with 16 Gb of RAM using Matlab R2022a.

\subsection{Benchmark of the pole selection}
\label{sec:numerical-experiments:poles}

In this preliminary section we first compare the convergence of the rational Krylov method for the choices of poles discussed in Section~\ref{sec:rational-approx}. In Table~\ref{tab:test-matrices} we report information on test matrices we are going use for the test. 
\begin{table}[htbp]
    \centering
    \caption{Test matrices: Finite Difference discretization of the Laplacian, centered differences for the 1D case, 5-point stencil for the 2D case, and the linear Finite Element Method discussed in Section~\ref{sec:fem-spaces}. $\lambda_{\min}$, $\lambda_{\max}$ are the smallest and largest eigenvalues respectively.}
    \label{tab:test-matrices}
    \begin{tabular}{rlccc}
        \toprule
        \multicolumn{2}{c}{$A$: Matrix name}      &  Size & $[\lambda_{\min},\lambda_{\max}]$ & Figure\\ 
        \midrule
         1D & FD Laplacian &  2048 & $[2.3\times 10^{-6},4]$ & \ref{fig:poles-1d-lap}\\
         2D & FD Laplacian &  4096 & $[0.0047,7.9953]$ & \ref{fig:poles-2d-lap}\\
         2D & Linear FEM & 1028 & $[0.0253,5.9008]$ & \ref{fig:poles-fem-lap}\\
        \botrule
    \end{tabular}
\end{table}
These are all discretizations of the Laplacian, as our goal is to deal with the discretization of the wave equation discussed in Section~\ref{sec:mol-for-wave}. For all the poles choices we compute the relative error
\[
\frac{\| \operatorname{sinc}(A)\mathbf{v} - V_k \operatorname{sinc}(V_k^T A V_k) V_k^T \mathbf{v} \|}{\| \operatorname{sinc}(A)\mathbf{v} \|}, 
\]
with $V_k$ the matrix whose column spans the relevant rational Krylov subspace $\mathcal{Q}_k(A,\mathbf{v})$, and the matrix-argument $\operatorname{sinc}(A)$ function computed in MATLAB as ``$\operatorname{sinc}(A) \gets \verb|A\funm(A,@sin)|$''. From Figure~\ref{fig:poles-benchmark} we observe that the best results are given by poles obtained from the Padé diagonal approximant of $\operatorname{sinc}$ function here computed in a symbolic way; see Appendix~\ref{sec:sinc-pade-poles}. Due to the parity of the function and by the fact that the diagonals approximant of odd order coincide with the previous even ones, we report only the even cases. The other satisfactory results are obtained with the poles $\mathcal{E}_n$ given by the expansion discussed in Section~\ref{sec:exponential}, which are the complex zeros of the non orthogonal Laguerre polynomials.
\begin{figure}[htb]
    \centering
    \subfloat[1D FD Laplacian\label{fig:poles-1d-lap}]{
%
%
%
\definecolor{mycolor1}{rgb}{0.00000,0.44700,0.74100}%
\definecolor{mycolor2}{rgb}{0.85000,0.32500,0.09800}%
\definecolor{mycolor3}{rgb}{0.92900,0.69400,0.12500}%
\definecolor{mycolor4}{rgb}{0.49400,0.18400,0.55600}%
\definecolor{mycolor5}{rgb}{0.46600,0.67400,0.18800}%
\definecolor{mycolor6}{rgb}{0.30100,0.74500,0.93300}%
\begin{tikzpicture}

\begin{axis}[%
width=0.33\columnwidth,
height=1.3in,
at={(0\columnwidth,0in)},
scale only axis,
xmin=2,
xmax=22,
xlabel style={font=\color{white!15!black}},
xlabel={Number of poles},
ymode=log,
ymin=2.23054503883866e-11,
ymax=0.0495751514419387,
yminorticks=true,
ylabel style={font=\color{white!15!black}},
ylabel={Relative Error},
axis background/.style={fill=white},
axis x line*=bottom,
axis y line*=left,
xmajorgrids,
ymajorgrids,
yminorgrids,
legend style={legend cell align=left, align=left, draw=white!15!black}
]
\addplot [color=mycolor1, line width=2.0pt, mark=o, mark options={solid, mycolor1}]
  table[row sep=crcr]{%
2	0.0495751514419387\\
3	0.0220651726139946\\
4	0.00837859636837678\\
5	0.00251157645991459\\
6	0.000631366175689662\\
7	0.000136862859576147\\
8	2.61123120397544e-05\\
9	4.67692342337327e-06\\
10	7.61468586013128e-07\\
11	1.11577351439746e-07\\
12	1.49406097433467e-08\\
13	1.88110814330385e-09\\
14	2.36051179121638e-10\\
15	7.06947469684652e-11\\
};

\addplot [color=mycolor2, line width=2.0pt, mark=square*, mark options={solid, mycolor2}]
  table[row sep=crcr]{%
4	0.00357512355518266\\
6	0.000379881323259452\\
8	8.12643198323276e-06\\
10	8.73945109290021e-08\\
12	3.11662540660298e-09\\
14	8.08901554771535e-11\\
16	7.50148070991825e-11\\
};

\addplot [color=mycolor3, line width=2.0pt, mark=triangle*, mark options={solid, mycolor3}]
  table[row sep=crcr]{%
5	0.00147506404397597\\
7	2.70531855469897e-05\\
9	7.4266298227117e-08\\
11	5.87856029421194e-10\\
13	4.55575459460845e-11\\
15	8.64560335300326e-11\\
17	1.65614286481213e-10\\
19	3.18225361233277e-11\\
};

\addplot [color=mycolor4, line width=2.0pt, mark=x, mark options={solid, mycolor4}]
  table[row sep=crcr]{%
2	0.0157466231065267\\
4	0.000144774338112685\\
6	3.1546171193386e-07\\
8	2.36823847141114e-10\\
10	7.3700080842856e-11\\
12	7.29617140595032e-11\\
14	7.33461293038416e-11\\
16	6.76089858768978e-11\\
18	7.26637145748342e-11\\
20	5.79402271411048e-11\\
};

\end{axis}
\end{tikzpicture}%
}
    \subfloat[2D FD Laplacian\label{fig:poles-2d-lap}]{
%
%
%
\definecolor{mycolor1}{rgb}{0.00000,0.44700,0.74100}%
\definecolor{mycolor2}{rgb}{0.85000,0.32500,0.09800}%
\definecolor{mycolor3}{rgb}{0.92900,0.69400,0.12500}%
\definecolor{mycolor4}{rgb}{0.49400,0.18400,0.55600}%
\definecolor{mycolor5}{rgb}{0.46600,0.67400,0.18800}%
\definecolor{mycolor6}{rgb}{0.30100,0.74500,0.93300}%
\begin{tikzpicture}

\begin{axis}[%
width=0.33\columnwidth,
height=1.3in,
at={(0\columnwidth,0in)},
scale only axis,
xmin=2,
xmax=22,
xlabel style={font=\color{white!15!black}},
xlabel={Number of poles},
ymode=log,
ymin=1.09426020442448e-13,
ymax=0.140482853660714,
yminorticks=true,
ylabel style={font=\color{white!15!black}},
ylabel={Relative Error},
axis background/.style={fill=white},
axis x line*=bottom,
axis y line*=left,
xmajorgrids,
ymajorgrids,
yminorgrids,
legend style={legend cell align=left, align=left, draw=white!15!black}
]
\addplot [color=mycolor1, line width=2.0pt, mark=o, mark options={solid, mycolor1}]
  table[row sep=crcr]{%
2	0.140482853660714\\
3	0.100247926221694\\
4	0.0622971271010977\\
5	0.0393895031687774\\
6	0.0228573249376433\\
7	0.0110735177298011\\
8	0.00501403089805743\\
9	0.00196862635061023\\
10	0.000680986125434816\\
11	0.000217926591842377\\
12	6.31018670101418e-05\\
13	1.6703929430778e-05\\
14	4.17775602463475e-06\\
15	9.68707507509574e-07\\
};

\addplot [color=mycolor2, line width=2.0pt, mark=square*, mark options={solid, mycolor2}]
  table[row sep=crcr]{%
4	0.0335591387795924\\
6	0.0100933243057053\\
8	0.00193144183828193\\
10	0.000125715352209563\\
12	2.4277300116361e-05\\
14	1.29072879281996e-06\\
16	1.97104790062211e-08\\
};

\addplot [color=mycolor3, line width=2.0pt, mark=triangle*, mark options={solid, mycolor3}]
  table[row sep=crcr]{%
5	0.0150452580249641\\
7	0.00304461428838038\\
9	8.05649041663936e-05\\
11	5.93994344489576e-06\\
13	9.77463991561725e-08\\
15	4.14642127565902e-10\\
17	7.52700931505455e-12\\
19	1.09426020442448e-13\\
};

\addplot [color=mycolor4, line width=2.0pt, mark=x, mark options={solid, mycolor4}]
  table[row sep=crcr]{%
2	0.0902505499008097\\
4	0.0121690184660869\\
6	0.000593078918657477\\
8	9.6321172724386e-06\\
10	6.01321237609366e-08\\
12	1.79993841322537e-10\\
14	4.52122814635133e-13\\
16	1.89577666655437e-13\\
18	2.2646091818289e-13\\
20	3.83329680404901e-13\\
};

\end{axis}
\end{tikzpicture}%
}

    \subfloat[2D Linear FEM\label{fig:poles-fem-lap}]{
%
%
%
\definecolor{mycolor1}{rgb}{0.00000,0.44700,0.74100}%
\definecolor{mycolor2}{rgb}{0.85000,0.32500,0.09800}%
\definecolor{mycolor3}{rgb}{0.92900,0.69400,0.12500}%
\definecolor{mycolor4}{rgb}{0.49400,0.18400,0.55600}%
\definecolor{mycolor5}{rgb}{0.46600,0.67400,0.18800}%
\definecolor{mycolor6}{rgb}{0.30100,0.74500,0.93300}%
\begin{tikzpicture}

\begin{axis}[%
width=0.33\columnwidth,
height=1.3in,
at={(0\columnwidth,0in)},
scale only axis,
xmin=2,
xmax=22,
xlabel style={font=\color{white!15!black}},
xlabel={Number of poles},
ymode=log,
ymin=6.18720580635138e-15,
ymax=0.111544813157411,
yminorticks=true,
ylabel style={font=\color{white!15!black}},
ylabel={Relative Error},
axis background/.style={fill=white},
axis x line*=bottom,
axis y line*=left,
xmajorgrids,
ymajorgrids,
yminorgrids,
legend style={legend cell align=left, align=left, draw=none, fill=none,at={(3in,0in)}}
]
\addplot [color=mycolor1, line width=2.0pt, mark=o, mark options={solid, mycolor1}]
  table[row sep=crcr]{%
2	0.111544813157411\\
3	0.0628180522147167\\
4	0.0321589677620976\\
5	0.0134396101115837\\
6	0.00500757665023514\\
7	0.00159955907942345\\
8	0.000491356293598763\\
9	0.000131045644021065\\
10	3.23334131986041e-05\\
11	6.84868671251421e-06\\
12	1.44367714627154e-06\\
13	2.76314279016716e-07\\
14	4.65003800227469e-08\\
15	7.60073120484164e-09\\
};
\addlegendentry{$\mathcal{L}_n$}

\addplot [color=mycolor2, line width=2.0pt, mark=square*, mark options={solid, mycolor2}]
  table[row sep=crcr]{%
4	0.0171929325866321\\
6	0.00184404600452608\\
8	0.000102077011137601\\
10	1.18315027787482e-05\\
12	3.13507695526061e-07\\
14	4.09591480935796e-09\\
16	1.80245107627297e-10\\
};
\addlegendentry{$\overline{\mathcal{L}}_n$}

\addplot [color=mycolor3, line width=2.0pt, mark=triangle*, mark options={solid, mycolor3}]
  table[row sep=crcr]{%
5	0.00671383760716412\\
7	0.000113493947749235\\
9	6.70285916856751e-06\\
11	7.89504242293594e-08\\
13	2.23384622243327e-10\\
15	2.24403484114417e-12\\
17	1.36110082213692e-14\\
19	6.8824382935789e-15\\
};
\addlegendentry{$\mathcal{E}_n$}

\addplot [color=mycolor4, line width=2.0pt, mark=x, mark options={solid, mycolor4}]
  table[row sep=crcr]{%
2	0.0485729761677972\\
4	0.00160089273335496\\
6	1.38723303211639e-05\\
8	5.56601359619185e-08\\
10	9.32054453676647e-11\\
12	7.52858889020868e-14\\
14	8.47233968892243e-15\\
16	6.18720580635138e-15\\
18	1.03170567169617e-14\\
20	7.30839184473166e-15\\
};
\addlegendentry{{$[n/n]$}-Padé}

\end{axis}
\end{tikzpicture}%
}
    \caption{Relative error for the computation of $\operatorname{sinc}(A)\mathbf{v}$ with $A$ the matrices described in Table~\ref{tab:test-matrices} and $\mathbf{v}$ a randomly generated vector. We test all the pole choices described in Section~\ref{sec:rational-approx}.}
    \label{fig:poles-benchmark}
\end{figure}
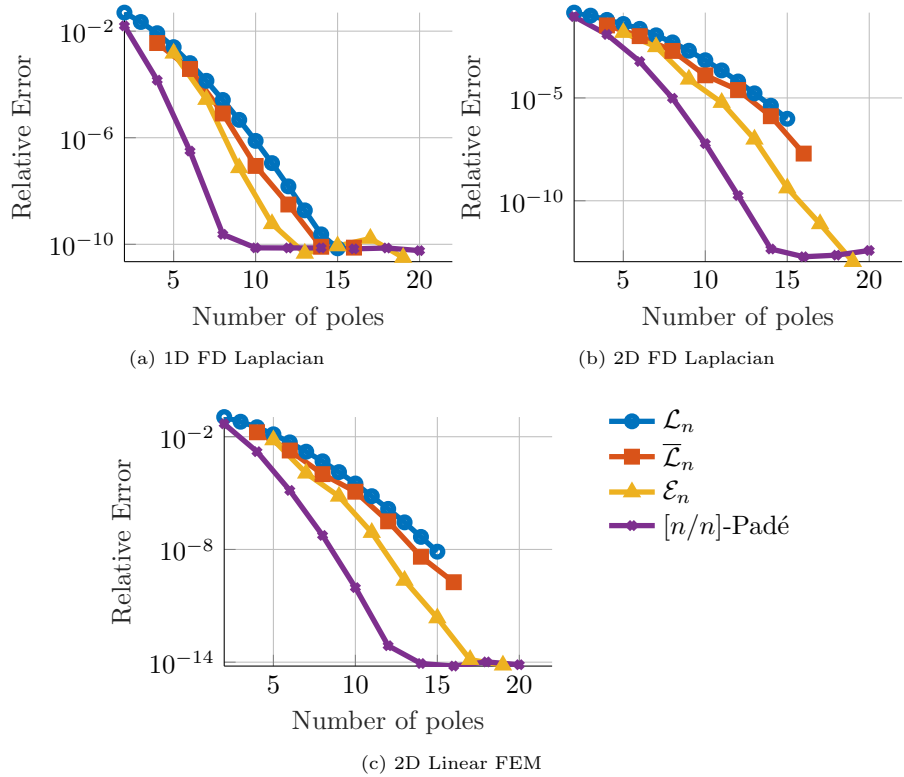
Also observe that the choice of poles as in $\overline{\mathcal{L}}_n,$ which preserve the symmetry in the approximation, slightly improves the convergence obtained using the poles of $\mathcal{L}_n.$ 

In the next set of experiments we compare the attained accuracy in terms of the relative error with the time needed to achieve it. In addition to rational Krylov-type approaches we consider also the exponential sums algorithm discussed in Section~\ref{sec:fourier-inversion}. To ensure that the strategy of approximating the matrix exponential within exponential sums does not reduce the overall accuracy we employ the cubic cost dense-matrix computation of the different exponential; this serves just as a sanity check of the overall procedure since it is indeed an expensive procedure. 
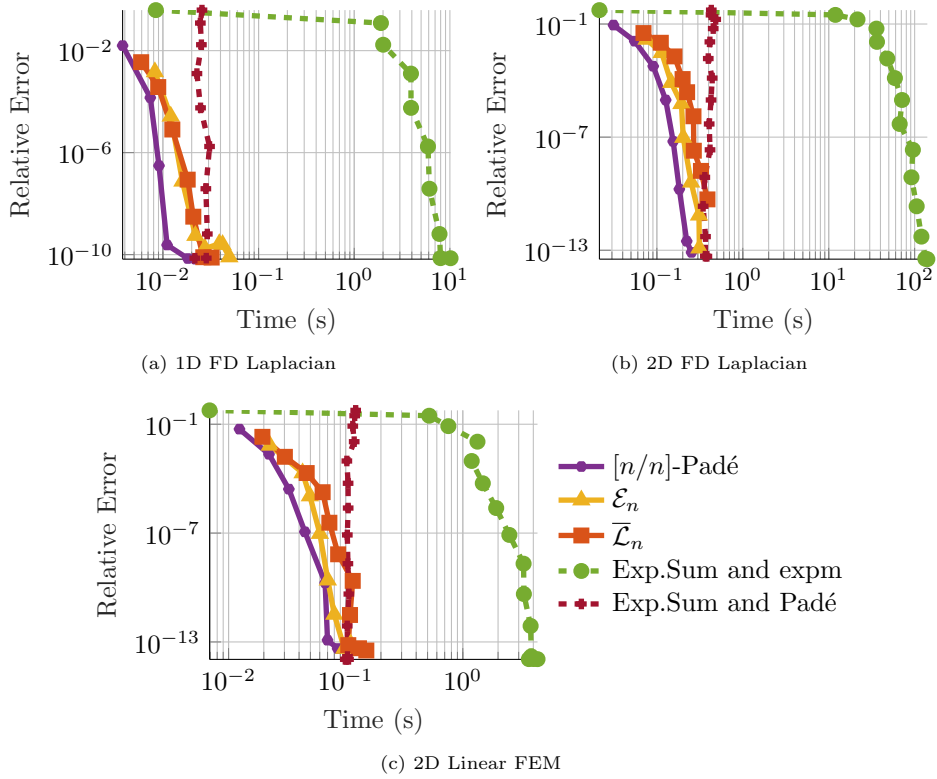
\begin{figure}[htbp]
    \centering

    \subfloat[1D FD Laplacian]{
%
%
%
\definecolor{mycolor1}{rgb}{0.49400,0.18400,0.55600}%
\definecolor{mycolor2}{rgb}{0.92900,0.69400,0.12500}%
\definecolor{mycolor3}{rgb}{0.85000,0.32500,0.09800}%
\definecolor{mycolor4}{rgb}{0.46600,0.67400,0.18800}%
\definecolor{mycolor5}{rgb}{0.63500,0.07800,0.18400}%
\begin{tikzpicture}

\begin{axis}[%
width=0.33\columnwidth,
height=1.3in,
at={(0\columnwidth,0in)},
scale only axis,
xmode=log,
xmin=0.003794,
xmax=10.15333,
xminorticks=true,
xlabel style={font=\color{white!15!black}},
xlabel={Time (s)},
ymode=log,
ymin=6.64919782409009e-11,
ymax=0.388786193440371,
yminorticks=true,
ylabel style={font=\color{white!15!black}},
ylabel={Relative Error},
axis background/.style={fill=white},
axis x line*=bottom,
axis y line*=left,
xmajorgrids,
xminorgrids,
ymajorgrids,
yminorgrids,
legend style={legend cell align=left, align=left, draw=white!15!black}
]
\addplot [color=mycolor1, line width=2.0pt, mark=asterisk, mark options={solid, mycolor1}]
  table[row sep=crcr]{%
0.003794	0.0157616903898085\\
0.007385	0.00014347522783652\\
0.009107	3.08999331047919e-07\\
0.011066	2.38882003094691e-10\\
0.018257	7.04121530115512e-11\\
0.021837	6.64919782409009e-11\\
};

\addplot [color=mycolor2, line width=2.0pt, mark=triangle*, mark options={solid, mycolor2}]
  table[row sep=crcr]{%
0.008248	0.00146458005133532\\
0.012079	2.66586542189185e-05\\
0.01601	7.51594774950872e-08\\
0.021864	5.58252973217184e-10\\
0.027487	1.73497786757354e-10\\
0.038856	2.65763847511351e-10\\
0.042067	2.46064372651223e-10\\
0.049461	8.06166642741207e-11\\
};

\addplot [color=mycolor3, line width=2.0pt, mark=square*, mark options={solid, mycolor3}]
  table[row sep=crcr]{%
0.00593	0.0035313385635004\\
0.00894	0.000373946853910692\\
0.012551	8.05688989222955e-06\\
0.018159	8.73964984844575e-08\\
0.0208	3.07893929761667e-09\\
0.026443	7.91590699833015e-11\\
0.03235	7.65644166557443e-11\\
};

\addplot [dashed, color=mycolor4, line width=2.0pt, mark=otimes*, mark options={solid, mycolor4}]
  table[row sep=crcr]{%
0.008444	0.388786193440371\\
1.901068	0.120789782490323\\
2.007263	0.0170452030391639\\
3.917559	0.00125868080077531\\
3.923221	5.69748840630372e-05\\
5.897356	1.74253692000863e-06\\
6.113754	3.84392570510959e-08\\
7.786305	6.44844487165247e-10\\
7.983531	7.27183624738443e-11\\
10.15333	7.22365687048356e-11\\
};

\addplot [dashed, color=mycolor5, line width=2.0pt, mark=+, mark options={solid, mycolor5}]
  table[row sep=crcr]{%
0.025624	0.388786193440371\\
0.024638	0.120789782490323\\
0.025269	0.0170452030391638\\
0.022692	0.00125868080077558\\
0.024755	5.69748840629833e-05\\
0.030744	1.74253692006269e-06\\
0.028122	3.84392570694149e-08\\
0.029149	6.4484439832067e-10\\
0.028411	7.2716936792572e-11\\
0.021576	7.22351791274704e-11\\
};

\end{axis}
\end{tikzpicture}%
}
    \subfloat[2D FD Laplacian]{
%
%
%
\definecolor{mycolor1}{rgb}{0.49400,0.18400,0.55600}%
\definecolor{mycolor2}{rgb}{0.92900,0.69400,0.12500}%
\definecolor{mycolor3}{rgb}{0.85000,0.32500,0.09800}%
\definecolor{mycolor4}{rgb}{0.46600,0.67400,0.18800}%
\definecolor{mycolor5}{rgb}{0.63500,0.07800,0.18400}%
\begin{tikzpicture}

\begin{axis}[%
width=0.33\columnwidth,
height=1.3in,
at={(0\columnwidth,0in)},
scale only axis,
xmode=log,
xmin=0.021451,
xmax=140.430961,
xminorticks=true,
xlabel style={font=\color{white!15!black}},
xlabel={Time (s)},
ymode=log,
ymin=3.3702404226666e-14,
ymax=0.54712206599271,
yminorticks=true,
ylabel style={font=\color{white!15!black}},
ylabel={Relative Error},
axis background/.style={fill=white},
axis x line*=bottom,
axis y line*=left,
xmajorgrids,
xminorgrids,
ymajorgrids,
yminorgrids,
legend style={legend cell align=left, align=left, draw=white!15!black}
]
\addplot [color=mycolor1, line width=2.0pt, mark=asterisk, mark options={solid, mycolor1}]
  table[row sep=crcr]{%
0.031293	0.0879625324717953\\
0.055253	0.012323033628343\\
0.090724	0.000577727595316761\\
0.126733	9.3434153657962e-06\\
0.15485	5.92853599013789e-08\\
0.182479	1.73225370330452e-10\\
0.222718	3.01198417720186e-13\\
0.253164	7.68871906175272e-14\\
};

\addplot [color=mycolor2, line width=2.0pt, mark=triangle*, mark options={solid, mycolor2}]
  table[row sep=crcr]{%
0.075819	0.0152271824601519\\
0.113786	0.00300215369321938\\
0.146823	7.77154752419822e-05\\
0.193288	5.89711211070612e-06\\
0.202689	9.48489533413838e-08\\
0.250529	4.02830768121151e-10\\
0.309209	7.17843177180665e-12\\
0.312326	1.28996133085166e-13\\
};

\addplot [color=mycolor3, line width=2.0pt, mark=square*, mark options={solid, mycolor3}]
  table[row sep=crcr]{%
0.069881	0.0332430001426952\\
0.111117	0.0101611346738331\\
0.160215	0.00191501719542422\\
0.199384	0.000120834490672698\\
0.224937	2.41790233945686e-05\\
0.266684	1.27653660346662e-06\\
0.269059	1.9041566296818e-08\\
0.326059	1.62282456609926e-09\\
0.387702	4.98591229892108e-11\\
};

\addplot [dashed, color=mycolor4, line width=2.0pt, mark=otimes*, mark options={solid, mycolor4}]
  table[row sep=crcr]{%
0.021451	0.547122065992709\\
11.960617	0.304883057158084\\
21.606291	0.176219647328301\\
35.94178	0.0572306153208974\\
36.441233	0.0112857791038884\\
48.022269	0.00147040808664599\\
58.902008	0.000135709384855721\\
71.320859	9.34081728010161e-06\\
66.949903	4.98312406995699e-07\\
95.349583	2.12255620604874e-08\\
91.795262	7.39035983673146e-10\\
105.578973	2.14381872321717e-11\\
119.904843	5.27381945702339e-13\\
133.361129	3.53968079677977e-14\\
140.430961	3.3702404226666e-14\\
};

\addplot [dashed, color=mycolor5, line width=2.0pt, mark=+, mark options={solid, mycolor5}]
  table[row sep=crcr]{%
0.429987	0.54712206599271\\
0.427273	0.304883057158085\\
0.484121	0.176219647328301\\
0.45277	0.0572306153208974\\
0.413981	0.0112857791038883\\
0.396488	0.00147040808664616\\
0.440404	0.000135709384855676\\
0.427477	9.3408172800564e-06\\
0.40631	4.98312406912302e-07\\
0.415727	2.12255621257876e-08\\
0.361576	7.3903600849816e-10\\
0.345674	2.14383062996946e-11\\
0.368357	5.28395687275071e-13\\
0.366688	4.86269898154153e-14\\
0.382167	4.73866494679749e-14\\
};

\end{axis}
\end{tikzpicture}%
}

    \subfloat[2D Linear FEM]{
%
%
%
\definecolor{mycolor1}{rgb}{0.49400,0.18400,0.55600}%
\definecolor{mycolor2}{rgb}{0.92900,0.69400,0.12500}%
\definecolor{mycolor3}{rgb}{0.85000,0.32500,0.09800}%
\definecolor{mycolor4}{rgb}{0.46600,0.67400,0.18800}%
\definecolor{mycolor5}{rgb}{0.63500,0.07800,0.18400}%
\begin{tikzpicture}

\begin{axis}[%
width=0.33\columnwidth,
height=1.3in,
at={(0\columnwidth,0in)},
scale only axis,
xmode=log,
xmin=0.006888,
xmax=4.343616,
xminorticks=true,
xlabel style={font=\color{white!15!black}},
xlabel={Time (s)},
ymode=log,
ymin=1.09069169372876e-14,
ymax=0.581338386289497,
yminorticks=true,
ylabel style={font=\color{white!15!black}},
ylabel={Relative Error},
axis background/.style={fill=white},
axis x line*=bottom,
axis y line*=left,
xmajorgrids,
xminorgrids,
ymajorgrids,
yminorgrids,
legend style={at={(1.03,0.5)}, anchor=west, legend cell align=left, align=left, draw=none, fill=none}
]
\addplot [color=mycolor1, line width=2.0pt, mark=asterisk, mark options={solid, mycolor1}]
  table[row sep=crcr]{%
0.012392	0.0537371717615407\\
0.02214	0.00220831260381268\\
0.032743	2.64254461648703e-05\\
0.044833	1.17637603446604e-07\\
0.066758	1.81758623209436e-10\\
0.069984	1.25182399506367e-13\\
0.085307	4.35269200694126e-14\\
};
\addlegendentry{$[n/n]$-Padé}

\addplot [color=mycolor2, line width=2.0pt, mark=triangle*, mark options={solid, mycolor2}]
  table[row sep=crcr]{%
0.022517	0.00681423431867943\\
0.042864	0.000205444661980064\\
0.049504	1.1029175762729e-05\\
0.060182	9.35806276113702e-08\\
0.070386	2.53540181400066e-10\\
0.079997	3.27616389793453e-12\\
0.097355	4.32842398346524e-14\\
0.110325	7.67173047258713e-14\\
};
\addlegendentry{$\mathcal{E}_n$}

\addplot [color=mycolor3, line width=2.0pt, mark=square*, mark options={solid, mycolor3}]
  table[row sep=crcr]{%
0.019458	0.0202504569211207\\
0.030254	0.00159909942958877\\
0.046454	0.000204106199752694\\
0.06357	1.79495468344519e-05\\
0.072572	3.72803789760571e-07\\
0.086288	6.76381656614158e-09\\
0.114869	2.31949930246838e-10\\
0.108893	3.05624449878925e-12\\
0.105226	6.68839262086467e-14\\
0.129859	4.49117826015729e-14\\
0.150057	3.36747983675824e-14\\
};
\addlegendentry{$\overline{\mathcal{L}}_n$}

\addplot [dashed, color=mycolor4, line width=2.0pt, mark=otimes*, mark options={solid, mycolor4}]
  table[row sep=crcr]{%
0.006888	0.581338386289497\\
0.516799	0.294157550349406\\
0.752824	0.0784412320569056\\
1.334127	0.0109067828042546\\
1.190968	0.000940827267081357\\
1.481789	5.55666719108095e-05\\
1.927981	2.39674935677947e-06\\
2.497389	7.90116679805137e-08\\
3.303733	2.05959518715406e-09\\
3.322251	4.35905751798882e-11\\
3.793498	7.65248126355237e-13\\
3.842785	1.59820857852273e-14\\
3.666332	1.13122821107538e-14\\
4.343616	1.12586753814205e-14\\
4.012161	1.14001040440783e-14\\
};
\addlegendentry{Exp.Sum and expm}

\addplot [dashed, color=mycolor5, line width=2.0pt, mark=+, mark options={solid, mycolor5}]
  table[row sep=crcr]{%
0.121715	0.581338386289497\\
0.119359	0.294157550349406\\
0.114968	0.0784412320569058\\
0.118223	0.0109067828042546\\
0.102331	0.000940827267081592\\
0.104855	5.55666719106683e-05\\
0.105014	2.39674935688485e-06\\
0.103288	7.90116678494592e-08\\
0.104435	2.05959528880941e-09\\
0.106422	4.3590480870454e-11\\
0.103128	7.65365856039104e-13\\
0.103601	1.56656394256099e-14\\
0.105738	1.09725380563132e-14\\
0.101495	1.09069169372876e-14\\
0.099268	1.10343357397368e-14\\
};
\addlegendentry{Exp.Sum and Padé}

\end{axis}
\end{tikzpicture}%
}
    
    \caption{Computation of the exponential sums~\eqref{eq:fourierapprox} obtained using different approximations for the exponential, relative error versus time (s) graphs. To have a comparison with the results in Figure~\ref{fig:poles-benchmark} we report in both pictures the results obtained with the $[n/n]$-Padé approximation, and the rational Krylov methods with poles $\mathcal{E}_n$, and $\overline{\mathcal{L}}_n$. The number of poles for the $[k/k]$-Padé approximation of the exponential is $k = 15$ for the 1D FD case and $k = 20$ for the 2D FD and Linear FEM cases. The number of Gauss-Legendre quadrature nodes goes from $\nu = 1,\ldots,15$ in all the cases.}
    \label{fig:exponential_sum}
\end{figure}
From the results in Figure~\ref{fig:exponential_sum} we observe that all the approximation methods reach the same accuracy for a comparable number of nodes, and that the pure rational Krylov strategies are the most cost effective.

\subsection{A synthetic example}\label{sec:a-synthetic-example}

In this section we consider a synthetic example of an equation of the form~\eqref{eq:theproblem}. In particular we consider as matrix $A = T_N T_N^T$ with $T_N$ the Rutishauser matrix, i.e., the Toeplitz matrix of size $N$ with eigenvalues in the complex plane which are approximately on the curve $ 2 \cos(2\theta) + 20 i \sin(\theta) $, $\theta \in [0,2\pi]$. As forcing term $f$ we consider the function $f(t) = \frac12 \sin(t)$. The initial conditions are given by the vector identical to one for the positions and by the null vector for the velocities. To have a reference solution, we use the MATLAB integrator \texttt{ode15s} with absolute and relative tolerances equal to~$10^{-12}$ and $3 \times 10^{-14}$, respectively. With this set of experiments firstly we want to verify that the order of Gautschi's method is preserved, secondly that it is possible to use the error analysis we have done to properly choose the number of poles.
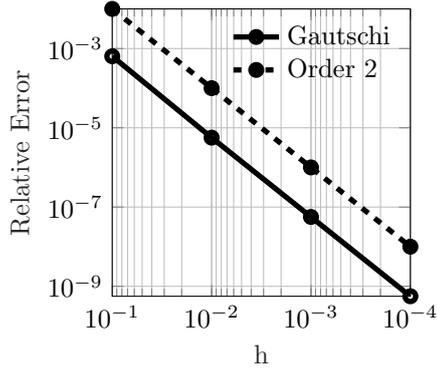
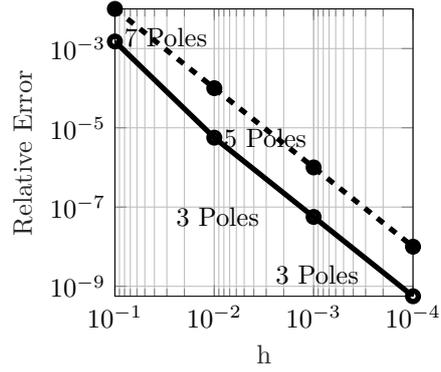
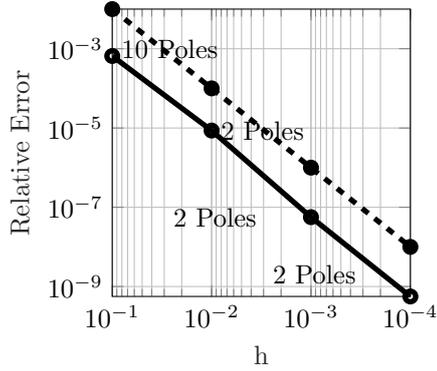
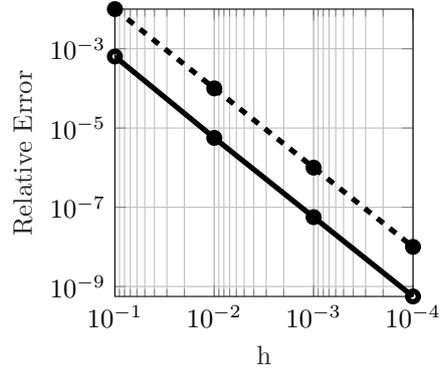
\begin{figure}[htbp]
    \centering
    \subfloat[Dense-matrix function computation. \label{fig:convergence-full}]{
%
%
\begin{tikzpicture}

\begin{axis}[%
width=0.3\columnwidth,
height=1.5in,
at={(0.758in,0.481in)},
scale only axis,
x dir=reverse,
xmode=log,
xmin=0.0001,
xmax=0.1,
xminorticks=true,
xlabel style={font=\color{white!15!black}},
xlabel={h},
ymode=log,
ymin=5.61474329237679e-10,
ymax=0.01,
yminorticks=true,
ylabel style={font=\color{white!15!black}},
ylabel={Relative Error},
axis background/.style={fill=white},
xmajorgrids,
xminorgrids,
ymajorgrids,
yminorgrids,
legend style={legend cell align=left, align=left, draw=none, fill=none}
]
\addplot [color=black, line width=2.0pt, mark=o, mark options={solid, black}]
  table[row sep=crcr]{%
0.1	0.000637722825938219\\
0.01	5.63490989498073e-06\\
0.001	5.63116318690235e-08\\
0.0001	5.61474329237679e-10\\
};
\addlegendentry{Gautschi}

\addplot [color=black, dashed, line width=2.0pt, mark=*, mark options={solid, black}]
  table[row sep=crcr]{%
0.1	0.01\\
0.01	0.0001\\
0.001	1e-06\\
0.0001	1e-08\\
};
\addlegendentry{Order 2}
\end{axis}
\end{tikzpicture}%
}\hfil
    \subfloat[Poles obtained from the rational approximation~\eqref{eq:sinc-exp-approximation}.\label{fig:convergence-ratexp}]{
%
%
\begin{tikzpicture}

\begin{axis}[%
width=0.3\columnwidth,
height=1.5in,
at={(0.758in,0.481in)},
scale only axis,
x dir=reverse,
xmode=log,
xmin=0.0001,
xmax=0.1,
xminorticks=true,
xlabel style={font=\color{white!15!black}},
xlabel={h},
ymode=log,
ymin=5.61440672611356e-10,
ymax=0.01,
yminorticks=true,
ylabel style={font=\color{white!15!black}},
ylabel={Relative Error},
axis background/.style={fill=white},
xmajorgrids,
xminorgrids,
ymajorgrids,
yminorgrids,
legend style={legend cell align=left, align=left, draw=white!15!black}
]
\addplot [color=black, line width=2.0pt, mark=o, mark options={solid, black}]
  table[row sep=crcr]{%
0.1	0.00150740226573584\\
0.01	5.63509764202648e-06\\
0.001	5.63276395714019e-08\\
0.0001	5.61440672611356e-10\\
};

\addplot [color=black, dashed, line width=2.0pt, mark=*, mark options={solid, black}]
  table[row sep=crcr]{%
0.1	0.01\\
0.01	0.0001\\
0.001	1e-06\\
0.0001	1e-08\\
};

\node[right, align=left]
at (axis cs:0.1,0.002) {  7 Poles};
\node[right, align=left]
at (axis cs:0.01,5.63490958128422e-06) {  5 Poles};
\node[right, align=left]
at (axis cs:0.03,5.6311406101669e-08) {  3 Poles};
\node[right, align=left]
at (axis cs:0.003,2e-9) {  3 Poles};
\end{axis}
\end{tikzpicture}%
}

    \subfloat[Poles obtained from the rational approximation~\eqref{eq:approx-pade-laguerre1}.\label{fig:convergence-rathyp}]{
%
%
\begin{tikzpicture}

\begin{axis}[%
width=0.3\columnwidth,
height=1.5in,
at={(0.772in,0.473in)},
scale only axis,
x dir=reverse,
xmode=log,
xmin=0.0001,
xmax=0.1,
xminorticks=true,
xlabel style={font=\color{white!15!black}},
xlabel={h},
ymode=log,
ymin=5.61456045597667e-10,
ymax=0.01,
yminorticks=true,
ylabel style={font=\color{white!15!black}},
ylabel={Relative Error},
axis background/.style={fill=white},
xmajorgrids,
xminorgrids,
ymajorgrids,
yminorgrids,
legend style={legend cell align=left, align=left, draw=white!15!black}
]
\addplot [color=black, line width=2.0pt, mark=o, mark options={solid, black}]
  table[row sep=crcr]{%
0.1	0.000656045326286969\\
0.01	8.62187703309163e-06\\
0.001	5.63096588041614e-08\\
0.0001	5.61456045597667e-10\\
};

\addplot [color=black, dashed, line width=2.0pt, mark=*, mark options={solid, black}]
  table[row sep=crcr]{%
0.1	0.01\\
0.01	0.0001\\
0.001	1e-06\\
0.0001	1e-08\\
};

\node[right, align=left]
at (axis cs:0.1,0.001) {  10 Poles};
\node[right, align=left]
at (axis cs:0.01,8.62187702633534e-06) {  2 Poles};
\node[right, align=left]
at (axis cs:0.03,5.63096700255167e-08) {  2 Poles};
\node[right, align=left]
at (axis cs:0.003,2e-9) {  2 Poles};
\end{axis}
\end{tikzpicture}%
}\hfil
    \subfloat[Exponential sum algorithm with 10 poles for the rational Krylov approximation of the exponential, and 10 quadrature points to generate the exponential sum.\label{fig:convergence-expsum}]{
%
%
\begin{tikzpicture}

\begin{axis}[%
width=0.3\columnwidth,
height=1.5in,
at={(0.772in,0.473in)},
scale only axis,
x dir=reverse,
xmode=log,
xmin=0.0001,
xmax=0.1,
xminorticks=true,
xlabel style={font=\color{white!15!black}},
xlabel={h},
ymode=log,
ymin=5.61441665447545e-10,
ymax=0.01,
yminorticks=true,
ylabel style={font=\color{white!15!black}},
ylabel={Relative Error},
axis background/.style={fill=white},
xmajorgrids,
xminorgrids,
ymajorgrids,
yminorgrids,
legend style={legend cell align=left, align=left, draw=white!15!black}
]
\addplot [color=black, line width=2.0pt, mark=o, mark options={solid, black}]
  table[row sep=crcr]{%
0.1	0.00063772308539876\\
0.01	5.63490993101205e-06\\
0.001	5.63116268492841e-08\\
0.0001	5.61441665447545e-10\\
};

\addplot [color=black, dashed, line width=2.0pt, mark=*, mark options={solid, black}]
  table[row sep=crcr]{%
0.1	0.01\\
0.01	0.0001\\
0.001	1e-06\\
0.0001	1e-08\\
};

\end{axis}
\end{tikzpicture}%
}
    
    \cprotect\caption{Relative $2$-norm error at final time $T = 1$, for matrix size $N = 20$. We use the bounds from Proposition~\ref{pro:how-many-poles} to determine the number of poles we need to achieve a given tolerance. The $\| \cdot \|_\Sigma$ norm is estimated by evaluating the $\| \cdot \|_\infty$ bound on an equally spaced grid of the interval {$[0,h^2 \lambda_{\text{max}}(A)] = [0,h^2 \verb|1.2138e+03|]$}.}
\end{figure}
In Figure~\ref{fig:convergence-full} we observe that using the dense-matrix computation of the matrix functions, the method~\eqref{eq:lubich_scheme_compact} obtains the desired order of convergence. From Figures~\ref{fig:convergence-ratexp} and~\ref{fig:convergence-rathyp} we observe a notable consequence of the reduction of the time integration step $h$. If we try to obtain an increase in accuracy by reducing $h$, then we produce a scaling of the spectrum of the matrix~$A$. Namely the set $\Sigma$ on which compute the bounds shrinks and this makes the computation of the involved matrix function easier. In other words, we need a lower number of poles to achieve a higher precision via the reduction of the time step $h$. On the other hand, when we want to apply the method based on exponential sums in Figure~\ref{fig:convergence-expsum}, we have a fixed cost per step which is given by the generation of the rational Krylov space for the approximation of the exponential matrix-vector products. The accuracy with which we make this product limits the final accuracy of the quadrature formula, so to maintain the second order we need a number of poles, and therefore of linear system solutions, higher than the other two cases in which the errors combine more favorably. For all strategies the error analysis allows to maintain the order 2 of the method.

To have a comparison, we consider the exponential integrator Adams-Bashforth-N{\o}rsett scheme of stiff order 2 from~\cite{10.1145/1206040.1206044}. This method uses a direct computation of the involved $\varphi$ functions, thus we compare it against our implementation that also uses the direct computation. To enhance the difference with respect to the cost, i.e., having to compute a matrix function of a matrix of double the size, we consider the same test problem but with $N=100$. From the results in Table~\ref{tab:etd2rk} we observe that for a comparable error, having to compute smaller-dimensional matrix functions has the advantage in terms of expected time.
\begin{table}[htbp]
    \centering
    \begin{tabular}{rllll}
\toprule
  & \multicolumn{2}{c}{Gautschi}          & \multicolumn{2}{c}{Adams-Bashforth-N{\o}rsett}           \\
\midrule
$h$      & T (s) & Rel. Err. & T (s) & Rel. Err. \\
\midrule
1.0e-01 & 4.44e-02 & 1.12e-04 & 5.49e-01 & 3.76e-05 \\
1.0e-02 & 3.04e-02 & 1.12e-06 & 2.66e-01 & 4.69e-07 \\
1.0e-03 & 1.42e-02 & 4.71e-09 & 2.85e-01 & 4.80e-09 \\
1.0e-04 & 4.39e-01 & 6.18e-10 & 1.58e+00 & 2.35e-10 \\
\bottomrule
    \end{tabular}
    \caption{Comparison in terms of elapsed time and two-norm relative error of the Gautschi integrator and the Adams-Bashforth-N{\o}rsett scheme of stiff order 2 from~\cite{10.1145/1206040.1206044}.}
    \label{tab:etd2rk}
\end{table}

\subsection{Solution of the linear wave equation}
\label{sec:numerical-experiments:waves}

Let us now test the different strategies implemented for the computation of the $\operatorname{sinc}$ function to compute the matrix-vector products using the matrix functions from~\eqref{eq:stormerfunctions} in the scheme~\eqref{eq:lubich_scheme_compact}. We consider the wave equation~\eqref{eq:acoustic-wave} on the domain and triangular mesh depicted in the first two panels of Figure~\ref{fig:domain-and-mesh}. 
\begin{figure}[htbp]
    \centering
    \includegraphics[width=\columnwidth]{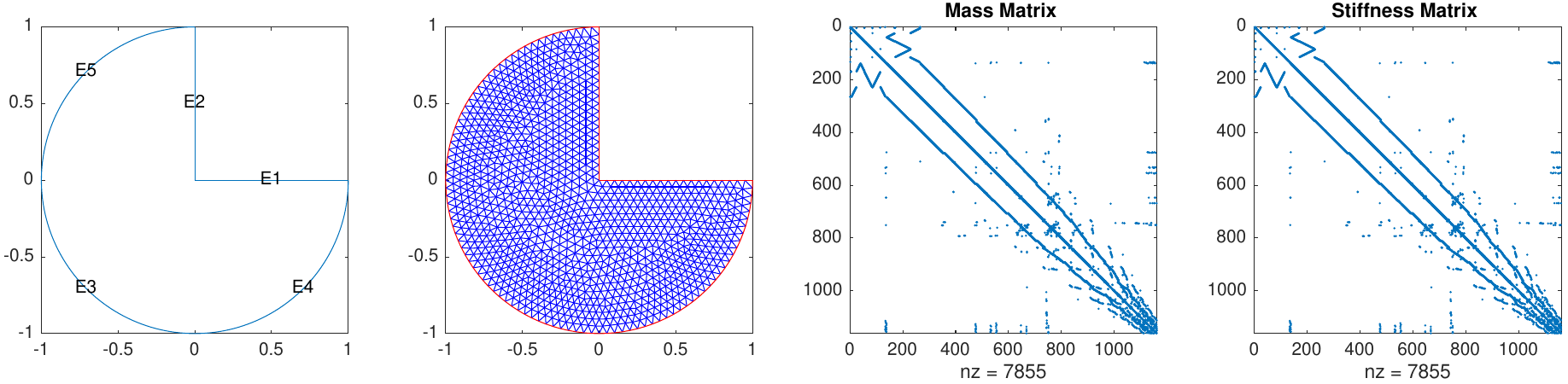}
    \caption{From left to right, we have the domain, the triangular mesh for the solution of the wave equation~\eqref{eq:acoustic-wave}, the mass and stiffness matrix for the $\mathbb{P}^1$-element on the associated triangulation.}
    \label{fig:domain-and-mesh}
\end{figure}
We construct a problem in which the boundary conditions are homogeneous Dirichlet conditions on the whole boundary, i.e., $\Gamma_N = \emptyset$. The initial data for the positions $\mathbf{x} = (x,y)$ is given by a perturbation of the shape
\[
u_0(x,y) = 0.8 \exp(-\nicefrac{(x+0.3)^2}{0.06} -\nicefrac{(y+0.3)^2}{0.06}),
\]
and a zero initial velocity, i.e., $v_0(x,y) = 0$. Since we simply want to test the robustness of the routines for computing the different matrix functions, we also set the forcing term equal to zero. 
\begin{figure}[htbp]
    \centering
    \subfloat[Direct]{\includegraphics[width=0.5\columnwidth]{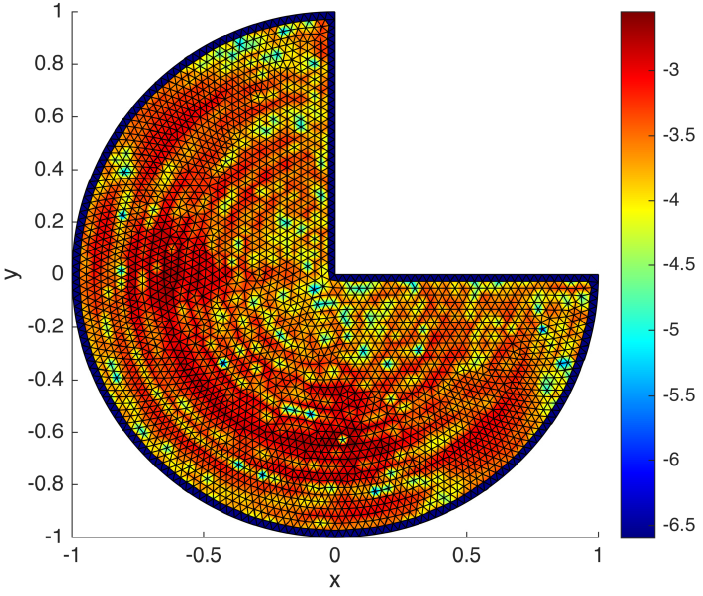}}
    \subfloat[$\mathcal{E}_n$]{\includegraphics[width=0.5\columnwidth]{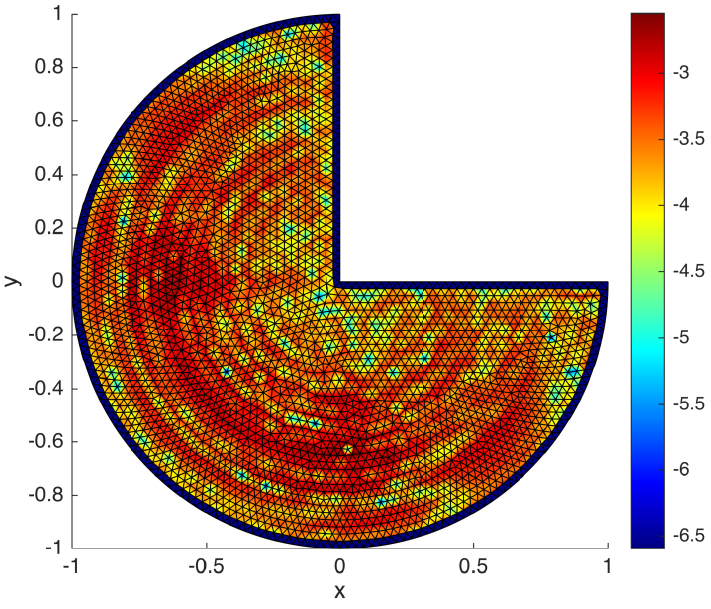}}
    
    \subfloat[$\overline{\mathcal{L}}_n$]{\includegraphics[width=0.5\columnwidth]{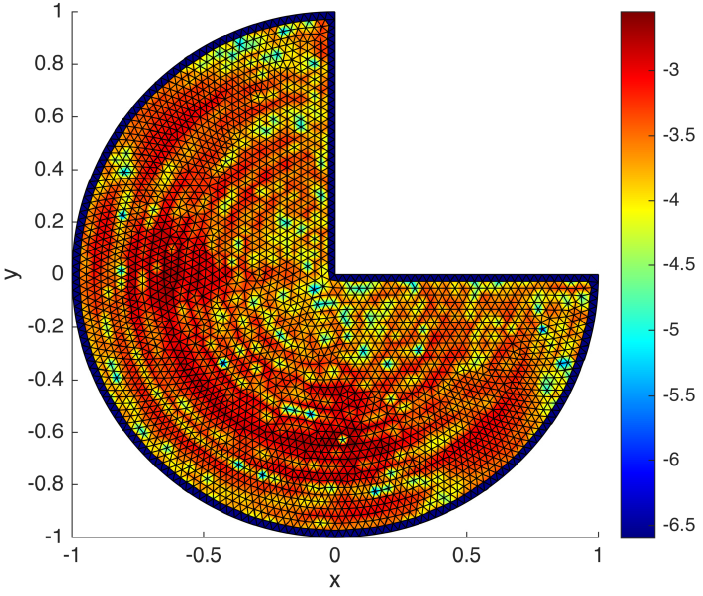}}
    \subfloat[Exponential sum]{\includegraphics[width=0.5\columnwidth]{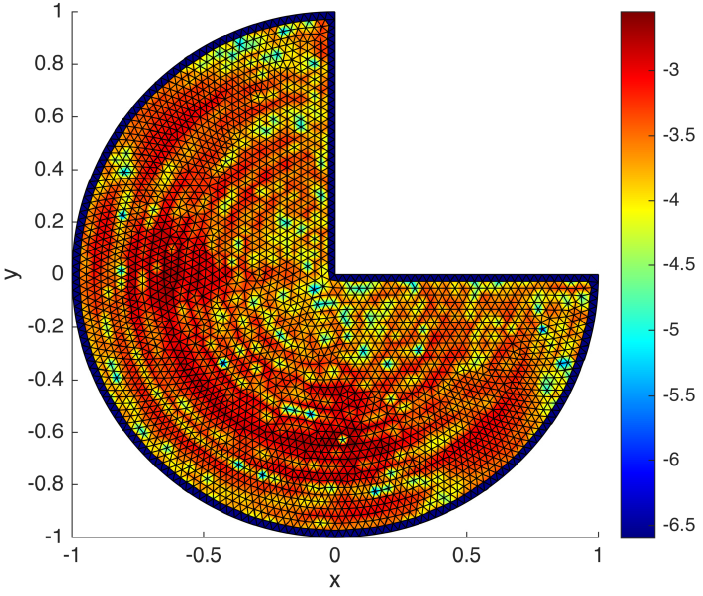}}
    
    \caption{Pointwise absolute error ($\log_{10}$-scale) computed against the solution generated with MATLAB's \texttt{solvepde} command at $T = 1$ with $h = 10^{-2}$ with respect to time, and linear FEM with largest element size of $h_{\text{max}} = 0.03$ in the \texttt{generateMesh} program {($3051$ inner dofs)}. We used four poles both for the $\mathcal{E}_n$ and $\overline{\mathcal{L}}_n$. The exponential sum case has been computed with 6 poles for the rational approximation of the exponential and 12 terms in the sum.}
    \label{fig:wave-eq-errors}
\end{figure}
From the errors reported in Figure~\ref{fig:wave-eq-errors} it can be observed that also in this case all the methods are capable of obtaining the same result as the direct method for a suitable (and limited) choice of the number of poles.

\section{Conclusion}
\label{sec:conclusions}

We have analyzed several strategies for the computation of matrix functions that appear in Gautschi-type trigonometric integrators for second-order differential equations. The analysis includes bounds to determine the number of poles of the rational approaches needed to achieve a specified accuracy. Future developments of the techniques discussed here concern the use of algorithms for the approximate solution of linear systems within rational Krylov spaces; and the possible application of the rational-Krylov methods with respect to the poles we have discussed here directly to exponential integrators for first-order systems. Furthermore, having observed the correlation between convergence rate of the method for the matrix function, amplitude of the time discretization step and global error, one should investigate the adaptive choice of the integration step and the relative adaptive choice of the number of poles to be used in the underlying rational approximation. Along the same lines, the possibility of also considering a space adaptive FEM discretization could further improve the compromise between accuracy and execution speed of the method.

\backmatter

\bmhead{Acknowledgments} The authors are members of the INdAM research group GNCS and thanks the two anonymous referees for their valuable comments and suggestions.

\section*{Declarations}

\bmhead{Funding} This work was partially supported by the ``GNCS Research Project - INdAM'' with code CUP\_E53C22001930001, and by the Spoke 1 ``FutureHPC \& BigData''  of the Italian Research Center on High-Performance Computing, Big Data and Quantum Computing (ICSC) funded by MUR Missione 4 Componente 2 Investimento 1.4: Potenziamento strutture di ricerca e creazione di ``campioni nazionali di R\&S (M4C2-19 )'' - Next Generation EU (NGEU).

\bmhead{Code availability} The code used to produce the results is available at \url{https://github.com/Cirdans-Home/rationalsecondorder}.

\begin{appendices}

\section{Denominators of Padé diagonal approximations}\label{sec:sinc-pade-poles}

Table~\ref{tab:diag-sinc} shows the denominators of the Padé expansions of the $\operatorname{sinc}(r)$ function for even degrees from 2 to 10 calculated in closed form with Mathematica (v. 12.2.0) and the {PadeApproximant} function:
\begin{center}
\begin{verbatim}
Table[CoefficientList[Denominator[PadeApproximant[Sinc[x], 
 {x, 0, n}]], x], {n, 2, 10, 2}]
\end{verbatim}
\end{center}
\begin{table}[htbp]
    \centering
        \caption{Coefficients in the monomial basis (ascending order) of the denominators of the diagonal Padé approximations for the $\operatorname{sinc}(r)$ function for the degree $n = 2m$, $m = 1,\ldots,5$.}
    \label{tab:diag-sinc}
    \begin{tabular}{c|p{0.8\columnwidth}}
    \toprule
    $n$ &  $[n/n]$-Padé denominators coefficients \\
    \midrule
2&$\{$ $1$, $0$, $\frac{1}{20}$ $\}$\\[0.5em]
4&$\{$ $1$, $0$, $\frac{13}{396}$, $0$, $\frac{5}{11088}\}$\\[0.5em]
6&$\{$ $1$, $0$, $\frac{1671}{69212}$, $0$, $\frac{97}{351384}$, $0$, $\frac{2623}{1644477120}$ $\}$\\[0.5em]
8&$\{$ $1$, $0$, $\frac{2290747}{120289892}$, $0$, $\frac{1281433}{7217393520}$, $0$, $\frac{560401}{562956694560}$, $0$, 
$\frac{1029037}{346781323848960}$ $\}$\\[0.5em]
10&$\{$ $1$, $0$, $\frac{34046903537}{2167379498676}$, $0$, $\frac{1679739379}{13726736824948}$, $0$,\\[0.5em] & $\frac{101555058991}{168015258737363520}$, $0$, $\frac{3924840709}{2016183104848362240}$, $0$, $\frac{37291724011}{11008359752472057830400}$ $\}$\\[0.5em]
    \botrule
    \end{tabular}
\end{table}
We stress that there exist a closed form expression of the diagonal Padé approximation of the $\operatorname{sinc}$ function that can be found in~\cite[p. 367]{MR367516} but involves computing determinants of matrices whose entries are binomial coefficients.

Computing the poles, i.e., the zeros, of such polynomials can be a delicate task. As the degree increases, the difference in absolute value of the largest and smallest coefficient drops below the machine precision. Where in general there are ad-hoc algorithms to deal with the computation of the zeros of polynomials that are difficult to represent, we resorted here in exploiting the symbolic functionalities of Mathematica to obtain a tabulation to machine precision of the poles up to degree $20$. The values are available in the Git repository; see Section~\ref{sec:numerical-experiments}.

\section{Depiction of the scalar bounds}\label{sec:appendix-bound}

We report here a depiction of the scalar bounds for the approximations~\eqref{eq:sinc-exp-approximation} and~\eqref{eq:approx-pade-laguerre1} on the real line. %
\begin{figure}[htbp]
    \centering
    \subfloat[Approximation based on the Padé expansion of the exponential~\eqref{eq:sinc-exp-approximation}\label{fig:sinc-exp-approximation}]{\includegraphics[width=0.4\columnwidth]{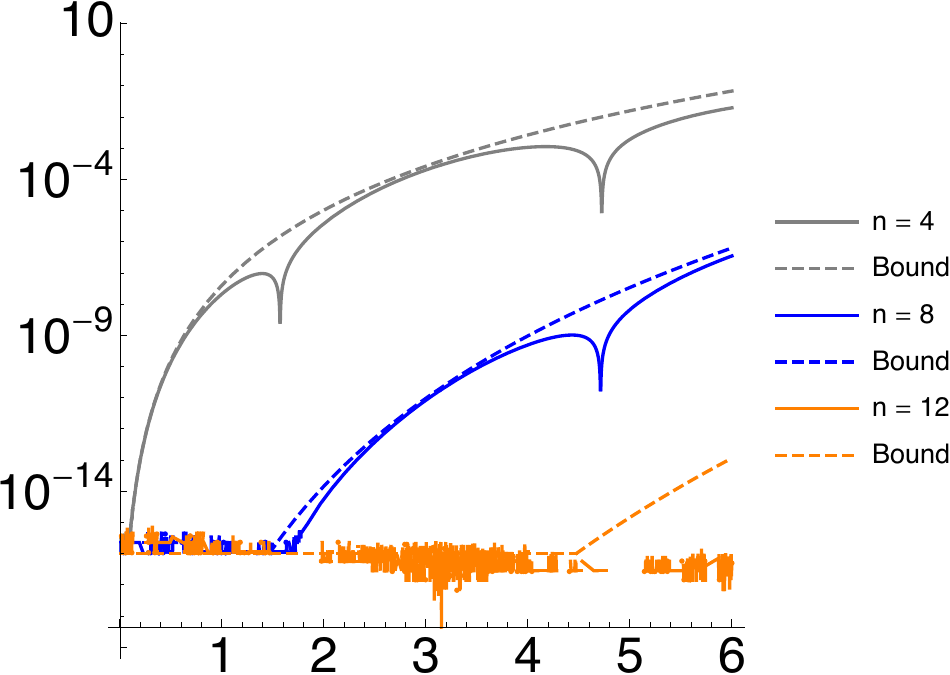}}\hfil
    \subfloat[Approximation based on the Padé expansion of the hypergeometric function~\eqref{eq:approx-pade-laguerre1}\label{fig:approx-pade-laguerre}]{\includegraphics[width=0.4\columnwidth]{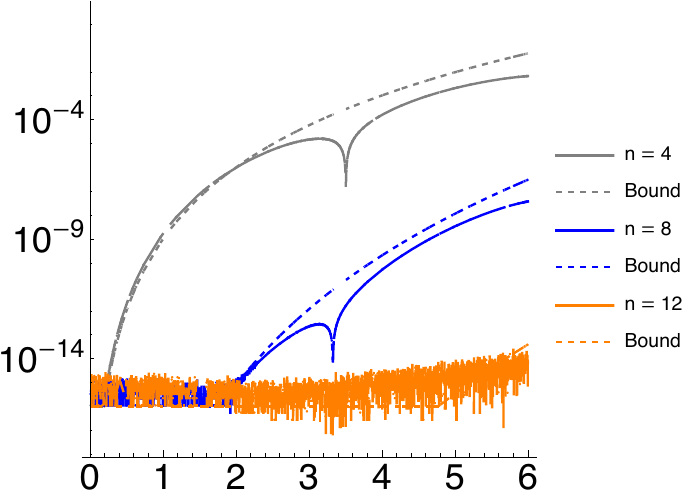}}
    \caption{Comparison of the theoretical bounds and the computed errors on the interval $[0,6]$ for the two rational approximation based on the Padé approximation of the exponential~\eqref{fig:sinc-exp-approximation} and the hypergeometric function~\eqref{fig:approx-pade-laguerre} discussed respectively in Sections~\ref{sec:exponential} and~\ref{sec:confhyperfunc}.}    \label{fig:bounds}
\end{figure}
As shown in the Figure~\ref{fig:bounds} there is an excellent correspondence between the theoretical bounds~\eqref{eq:sinc-exp-approximation-bound} and~\eqref{eq:approx-pade-laguerre-bound}  and the absolute error calculated with 16 significant figures. Furthermore, the results are also in good agreement with the behavior depicted in Figure~\ref{fig:poles-benchmark} with respect to the choice of poles for a given matrix $A$. Specifically, they confirm the behavior according to which the poles obtained by the exponential expansion return a better convergence than those obtained by the expansion based on the hypergeometric function.

\end{appendices}

\bibliography{bibliography}%


\begin{thebibliography}{25}
\ifx \bisbn   \undefined \def \bisbn  #1{ISBN #1}\fi
\ifx \binits  \undefined \def \binits#1{#1}\fi
\ifx \bauthor  \undefined \def \bauthor#1{#1}\fi
\ifx \batitle  \undefined \def \batitle#1{#1}\fi
\ifx \bjtitle  \undefined \def \bjtitle#1{#1}\fi
\ifx \bvolume  \undefined \def \bvolume#1{\textbf{#1}}\fi
\ifx \byear  \undefined \def \byear#1{#1}\fi
\ifx \bissue  \undefined \def \bissue#1{#1}\fi
\ifx \bfpage  \undefined \def \bfpage#1{#1}\fi
\ifx \blpage  \undefined \def \blpage #1{#1}\fi
\ifx \burl  \undefined \def \burl#1{\textsf{#1}}\fi
\ifx \doiurl  \undefined \def \doiurl#1{\url{https://doi.org/#1}}\fi
\ifx \betal  \undefined \def \betal{\textit{et al.}}\fi
\ifx \binstitute  \undefined \def \binstitute#1{#1}\fi
\ifx \binstitutionaled  \undefined \def \binstitutionaled#1{#1}\fi
\ifx \bctitle  \undefined \def \bctitle#1{#1}\fi
\ifx \beditor  \undefined \def \beditor#1{#1}\fi
\ifx \bpublisher  \undefined \def \bpublisher#1{#1}\fi
\ifx \bbtitle  \undefined \def \bbtitle#1{#1}\fi
\ifx \bedition  \undefined \def \bedition#1{#1}\fi
\ifx \bseriesno  \undefined \def \bseriesno#1{#1}\fi
\ifx \blocation  \undefined \def \blocation#1{#1}\fi
\ifx \bsertitle  \undefined \def \bsertitle#1{#1}\fi
\ifx \bsnm \undefined \def \bsnm#1{#1}\fi
\ifx \bsuffix \undefined \def \bsuffix#1{#1}\fi
\ifx \bparticle \undefined \def \bparticle#1{#1}\fi
\ifx \barticle \undefined \def \barticle#1{#1}\fi
\bibcommenthead
\ifx \bconfdate \undefined \def \bconfdate #1{#1}\fi
\ifx \botherref \undefined \def \botherref #1{#1}\fi
\ifx \url \undefined \def \url#1{\textsf{#1}}\fi
\ifx \bchapter \undefined \def \bchapter#1{#1}\fi
\ifx \bbook \undefined \def \bbook#1{#1}\fi
\ifx \bcomment \undefined \def \bcomment#1{#1}\fi
\ifx \oauthor \undefined \def \oauthor#1{#1}\fi
\ifx \citeauthoryear \undefined \def \citeauthoryear#1{#1}\fi
\ifx \endbibitem  \undefined \def \endbibitem {}\fi
\ifx \bconflocation  \undefined \def \bconflocation#1{#1}\fi
\ifx \arxivurl  \undefined \def \arxivurl#1{\textsf{#1}}\fi
\csname PreBibitemsHook\endcsname

\bibitem[\protect\citeauthoryear{Baker et~al.}{1977}]{MR448947}
\begin{barticle}
\bauthor{\bsnm{Baker}, \binits{G.A.}},
\bauthor{\bsnm{Bramble}, \binits{J.H.}},
\bauthor{\bsnm{Thom\'{e}e}, \binits{V.}}:
\batitle{Single step {G}alerkin approximations for parabolic problems}.
\bjtitle{Math. Comp.}
\bvolume{31}(\bissue{140}),
\bfpage{818}--\blpage{847}
(\byear{1977})
\doiurl{10.2307/2006116}
\end{barticle}
\endbibitem

\bibitem[\protect\citeauthoryear{Hochbruck and Ostermann}{2010}]{MR2652783}
\begin{barticle}
\bauthor{\bsnm{Hochbruck}, \binits{M.}},
\bauthor{\bsnm{Ostermann}, \binits{A.}}:
\batitle{Exponential integrators}.
\bjtitle{Acta Numer.}
\bvolume{19},
\bfpage{209}--\blpage{286}
(\byear{2010})
\doiurl{10.1017/S0962492910000048}
\end{barticle}
\endbibitem

\bibitem[\protect\citeauthoryear{Botchev
  et~al.}{0}]{https://doi.org/10.48550/arxiv.2206.06909}
\begin{barticle}
\bauthor{\bsnm{Botchev}, \binits{M.A.}},
\bauthor{\bsnm{Knizhnerman}, \binits{L.}},
\bauthor{\bsnm{Schweitzer}, \binits{M.}}:
\batitle{{Krylov Subspace Residual and Restarting for Certain Second Order
  Differential Equations}}.
\bjtitle{SIAM Journal on Scientific Computing}
\bvolume{0}(\bissue{0}),
\bfpage{223}--\blpage{253}
(\byear{0})
\doiurl{10.1137/22M1503300}
{\href{https://arxiv.org/abs/https://doi.org/10.1137/22M1503300}{{https://doi.org/10.1137/22M1503300}}}
\end{barticle}
\endbibitem

\bibitem[\protect\citeauthoryear{Hochbruck and Lubich}{1997}]{MR1472203}
\begin{barticle}
\bauthor{\bsnm{Hochbruck}, \binits{M.}},
\bauthor{\bsnm{Lubich}, \binits{C.}}:
\batitle{On {K}rylov subspace approximations to the matrix exponential
  operator}.
\bjtitle{SIAM J. Numer. Anal.}
\bvolume{34}(\bissue{5}),
\bfpage{1911}--\blpage{1925}
(\byear{1997})
\doiurl{10.1137/S0036142995280572}
\end{barticle}
\endbibitem

\bibitem[\protect\citeauthoryear{Tal-Ezer}{2007}]{MR2357621}
\begin{barticle}
\bauthor{\bsnm{Tal-Ezer}, \binits{H.}}:
\batitle{{On restart and error estimation for {K}rylov approximation
  of{$w=f(A)v$}}}.
\bjtitle{SIAM J. Sci. Comput.}
\bvolume{29}(\bissue{6}),
\bfpage{2426}--\blpage{2441}
(\byear{2007})
\doiurl{10.1137/040617868}
\end{barticle}
\endbibitem

\bibitem[\protect\citeauthoryear{Gautschi}{1961}]{MR138200}
\begin{barticle}
\bauthor{\bsnm{Gautschi}, \binits{W.}}:
\batitle{Numerical integration of ordinary differential equations based on
  trigonometric polynomials}.
\bjtitle{Numer. Math.}
\bvolume{3},
\bfpage{381}--\blpage{397}
(\byear{1961})
\doiurl{10.1007/BF01386037}
\end{barticle}
\endbibitem

\bibitem[\protect\citeauthoryear{Hochbruck and Lubich}{1999}]{MR1715573}
\begin{barticle}
\bauthor{\bsnm{Hochbruck}, \binits{M.}},
\bauthor{\bsnm{Lubich}, \binits{C.}}:
\batitle{A {G}autschi-type method for oscillatory second-order differential
  equations}.
\bjtitle{Numer. Math.}
\bvolume{83}(\bissue{3}),
\bfpage{403}--\blpage{426}
(\byear{1999})
\doiurl{10.1007/s002110050456}
\end{barticle}
\endbibitem

\bibitem[\protect\citeauthoryear{Botchev et~al.}{2006}]{MR2235388}
\begin{barticle}
\bauthor{\bsnm{Botchev}, \binits{M.A.}},
\bauthor{\bsnm{Harutyunyan}, \binits{D.}},
\bauthor{\bsnm{Vegt}, \binits{J.J.W.}}:
\batitle{{The {G}autschi time stepping scheme for edge finite element
  discretizations of the {M}axwell equations}}.
\bjtitle{J. Comput. Phys.}
\bvolume{216}(\bissue{2}),
\bfpage{654}--\blpage{686}
(\byear{2006})
\doiurl{10.1016/j.jcp.2006.01.014}
\end{barticle}
\endbibitem

\bibitem[\protect\citeauthoryear{Hairer et~al.}{1993}]{HNW1993}
\begin{bbook}
\bauthor{\bsnm{Hairer}, \binits{E.}},
\bauthor{\bsnm{N{\o}rsett}, \binits{S.P.}},
\bauthor{\bsnm{Wanner}, \binits{G.}}:
\bbtitle{Solving Ordinary Differential Equations. {I}: Nonstiff Problems},
\bedition{2}nd edn.
\bsertitle{Springer Series in Computational Mathematics},
vol. \bseriesno{8},
p. \bfpage{528}.
\bpublisher{Springer},
\blocation{Berlin, Heidelberg}
(\byear{1993})
\end{bbook}
\endbibitem

\bibitem[\protect\citeauthoryear{Bergermann and Stoll}{2024}]{MR4712819}
\begin{barticle}
\bauthor{\bsnm{Bergermann}, \binits{K.}},
\bauthor{\bsnm{Stoll}, \binits{M.}}:
\batitle{{Adaptive rational {K}rylov methods for exponential {R}unge-{K}utta
  integrators}}.
\bjtitle{SIAM J. Matrix Anal. Appl.}
\bvolume{45}(\bissue{1}),
\bfpage{744}--\blpage{770}
(\byear{2024})
\doiurl{10.1137/23M1559439}
\end{barticle}
\endbibitem

\bibitem[\protect\citeauthoryear{Bertaccini and Durastante}{[2021] \copyright
  2021}]{MR4266513}
\begin{bchapter}
\bauthor{\bsnm{Bertaccini}, \binits{D.}},
\bauthor{\bsnm{Durastante}, \binits{F.}}:
\bctitle{{Computing function of large matrices by a preconditioned rational
  {K}rylov method}}.
In: \bbtitle{Numerical Mathematics and Advanced Applications---{ENUMATH} 2019}.
\bsertitle{Lect. Notes Comput. Sci. Eng.},
vol. \bseriesno{139},
pp. \bfpage{343}--\blpage{351}.
\bpublisher{Springer},
\blocation{Cham}
(\byear{[2021] \copyright 2021}).
\doiurl{10.1007/978-3-030-55874-1\_33} .
\burl{https://doi.org/10.1007/978-3-030-55874-1_33}
\end{bchapter}
\endbibitem

\bibitem[\protect\citeauthoryear{G\"{u}ttel}{2013}]{MR3095912}
\begin{barticle}
\bauthor{\bsnm{G\"{u}ttel}, \binits{S.}}:
\batitle{Rational {K}rylov approximation of matrix functions: numerical methods
  and optimal pole selection}.
\bjtitle{GAMM-Mitt.}
\bvolume{36}(\bissue{1}),
\bfpage{8}--\blpage{31}
(\byear{2013})
\doiurl{10.1002/gamm.201310002}
\end{barticle}
\endbibitem

\bibitem[\protect\citeauthoryear{Crouzeix and Palencia}{2017}]{MR3666309}
\begin{barticle}
\bauthor{\bsnm{Crouzeix}, \binits{M.}},
\bauthor{\bsnm{Palencia}, \binits{C.}}:
\batitle{The numerical range is a {$(1+\sqrt{2})$}-spectral set}.
\bjtitle{SIAM J. Matrix Anal. Appl.}
\bvolume{38}(\bissue{2}),
\bfpage{649}--\blpage{655}
(\byear{2017})
\doiurl{10.1137/17M1116672}
\end{barticle}
\endbibitem

\bibitem[\protect\citeauthoryear{Magnus and Wynn}{1975}]{MR367516}
\begin{barticle}
\bauthor{\bsnm{Magnus}, \binits{A.}},
\bauthor{\bsnm{Wynn}, \binits{J.}}:
\batitle{On the {P}ad\'{e} table of {${\rm cos}z$}}.
\bjtitle{Proc. Amer. Math. Soc.}
\bvolume{47},
\bfpage{361}--\blpage{367}
(\byear{1975})
\doiurl{10.2307/2039747}
\end{barticle}
\endbibitem

\bibitem[\protect\citeauthoryear{Erd\'{e}lyi et~al.}{1953}]{Erdelyi-vol1}
\begin{bbook}
\bauthor{\bsnm{Erd\'{e}lyi}, \binits{A.}},
\bauthor{\bsnm{Magnus}, \binits{W.}},
\bauthor{\bsnm{Oberhettinger}, \binits{F.}},
\bauthor{\bsnm{Tricomi}, \binits{F.G.}}:
\bbtitle{Higher Transcendental Functions. {V}ol. {I}},
\bedition{1}st edn.,
p. \bfpage{302}.
\bpublisher{McGraw-Hill Book Co., Inc.},
\blocation{New York-Toronto-London}
(\byear{1953}).
\bcomment{Based, in part, on notes left by Harry Bateman}.
\burl{https://resolver.caltech.edu/CaltechAUTHORS:20140123-104529738}
\end{bbook}
\endbibitem

\bibitem[\protect\citeauthoryear{Luke}{1976}]{Luke}
\begin{botherref}
\oauthor{\bsnm{Luke}, \binits{Y.L.}}:
Algorithms for {R}ational {A}pproximations for a {C}onfluent {H}ypergeometric
  {F}unction {II}.
Interim rept. ADA032910,
Missouri University, Kansas City, Department of Mathematics,
\url{https://apps.dtic.mil/sti/citations/ADA032910}
(September 1976)
\end{botherref}
\endbibitem

\bibitem[\protect\citeauthoryear{Skaflestad and Wright}{2009}]{MR2492291}
\begin{barticle}
\bauthor{\bsnm{Skaflestad}, \binits{B.}},
\bauthor{\bsnm{Wright}, \binits{W.M.}}:
\batitle{The scaling and modified squaring method for matrix functions related
  to the exponential}.
\bjtitle{Appl. Numer. Math.}
\bvolume{59}(\bissue{3-4}),
\bfpage{783}--\blpage{799}
(\byear{2009})
\doiurl{10.1016/j.apnum.2008.03.035}
\end{barticle}
\endbibitem

\bibitem[\protect\citeauthoryear{}{2021}]{NIST:DLMF}
\begin{botherref}
{\it NIST Digital Library of Mathematical Functions}.
\url{http://dlmf.nist.gov/}, Release 1.1.3 of 2021-09-15.
F.~W.~J. Olver, A.~B. {Olde Daalhuis}, D.~W. Lozier, B.~I. Schneider, R.~F.
  Boisvert, C.~W. Clark, B.~R. Miller, B.~V. Saunders, H.~S. Cohl, and M.~A.
  McClain, eds.
(2021).
\url{http://dlmf.nist.gov/}
\end{botherref}
\endbibitem

\bibitem[\protect\citeauthoryear{Mart\'{\i}nez-Finkelshtein
  et~al.}{2001}]{MR1858305}
\begin{bchapter}
\bauthor{\bsnm{Mart\'{\i}nez-Finkelshtein}, \binits{A.}},
\bauthor{\bsnm{Mart\'{\i}nez-Gonz\'{a}lez}, \binits{P.}},
\bauthor{\bsnm{Orive}, \binits{R.}}:
\bctitle{On asymptotic zero distribution of {L}aguerre and generalized {B}essel
  polynomials with varying parameters}.
In: \bbtitle{Proceedings of the {F}ifth {I}nternational {S}ymposium on
  {O}rthogonal {P}olynomials, {S}pecial {F}unctions and Their {A}pplications
  ({P}atras, 1999)},
vol. \bseriesno{133},
pp. \bfpage{477}--\blpage{487}
(\byear{2001}).
\doiurl{10.1016/S0377-0427(00)00654-3} .
\burl{https://doi.org/10.1016/S0377-0427(00)00654-3}
\end{bchapter}
\endbibitem

\bibitem[\protect\citeauthoryear{Trefethen}{2008}]{MR2403058}
\begin{barticle}
\bauthor{\bsnm{Trefethen}, \binits{L.N.}}:
\batitle{Is {G}auss quadrature better than {C}lenshaw-{C}urtis?}
\bjtitle{SIAM Rev.}
\bvolume{50}(\bissue{1}),
\bfpage{67}--\blpage{87}
(\byear{2008})
\doiurl{10.1137/060659831}
\end{barticle}
\endbibitem

\bibitem[\protect\citeauthoryear{Kahaner et~al.}{1989}]{10.5555/61926}
\begin{bbook}
\bauthor{\bsnm{Kahaner}, \binits{D.}},
\bauthor{\bsnm{Moler}, \binits{C.}},
\bauthor{\bsnm{Nash}, \binits{S.}}:
\bbtitle{Numerical {M}ethods and {S}oftware}.
\bpublisher{Prentice-Hall, Inc.},
\blocation{USA}
(\byear{1989})
\end{bbook}
\endbibitem

\bibitem[\protect\citeauthoryear{Al-Mohy and Higham}{2011}]{MR2785959}
\begin{barticle}
\bauthor{\bsnm{Al-Mohy}, \binits{A.H.}},
\bauthor{\bsnm{Higham}, \binits{N.J.}}:
\batitle{Computing the action of the matrix exponential, with an application to
  exponential integrators}.
\bjtitle{SIAM J. Sci. Comput.}
\bvolume{33}(\bissue{2}),
\bfpage{488}--\blpage{511}
(\byear{2011})
\doiurl{10.1137/100788860}
\end{barticle}
\endbibitem

\bibitem[\protect\citeauthoryear{Schmelzer and Trefethen}{2007/08}]{MR2494943}
\begin{barticle}
\bauthor{\bsnm{Schmelzer}, \binits{T.}},
\bauthor{\bsnm{Trefethen}, \binits{L.N.}}:
\batitle{Evaluating matrix functions for exponential integrators via
  {C}arath\'{e}odory-{F}ej\'{e}r approximation and contour integrals}.
\bjtitle{Electron. Trans. Numer. Anal.}
\bvolume{29},
\bfpage{1}--\blpage{18}
(\byear{2007/08})
\end{barticle}
\endbibitem

\bibitem[\protect\citeauthoryear{Cohen and Pernet}{2017}]{MR3526765}
\begin{bbook}
\bauthor{\bsnm{Cohen}, \binits{G.}},
\bauthor{\bsnm{Pernet}, \binits{S.}}:
\bbtitle{Finite Element and Discontinuous {G}alerkin Methods for Transient Wave
  Equations}.
\bsertitle{Scientific Computation},
p. \bfpage{381}.
\bpublisher{Springer}, \blocation{???}
(\byear{2017}).
\doiurl{10.1007/978-94-017-7761-2} .
\bcomment{With a foreword by Patrick Joly}.
\burl{https://doi.org/10.1007/978-94-017-7761-2}
\end{bbook}
\endbibitem

\bibitem[\protect\citeauthoryear{Berland
  et~al.}{2007}]{10.1145/1206040.1206044}
\begin{barticle}
\bauthor{\bsnm{Berland}, \binits{H.}},
\bauthor{\bsnm{Skaflestad}, \binits{B.}},
\bauthor{\bsnm{Wright}, \binits{W.M.}}:
\batitle{{EXPINT}---{A} {MATLAB} {P}ackage for {E}xponential {I}ntegrators}.
\bjtitle{ACM Trans. Math. Softw.}
\bvolume{33}(\bissue{1}),
\bfpage{4}
(\byear{2007})
\doiurl{10.1145/1206040.1206044}
\end{barticle}
\endbibitem

\end{thebibliography}

\end{document}